\documentclass[12pt,a4paper,reqno]{amsart}

\usepackage[margin=2cm]{geometry}
\usepackage[T1]{fontenc}
\usepackage[utf8]{inputenc}

\usepackage{amssymb,amsmath,amsthm,amsfonts,mathrsfs}
\usepackage[foot]{amsaddr}
\usepackage{xcolor}
\usepackage{float}
\usepackage{tikz}
\usetikzlibrary{arrows.meta,decorations.pathreplacing}
\usepackage{pgfplots}
\pgfplotsset{compat=1.16}

\usepackage[
colorlinks=true,
urlcolor=blue,
citecolor=red,
linkcolor=blue,
linktocpage,
pdfpagelabels,
bookmarksnumbered,
bookmarksopen
]{hyperref}

\newtheorem{theor}{Theorem}[section]

\newtheorem{lem}{Lemma}[section]
\newtheorem{prop}{Proposition}[section]
\newtheorem{cor}{Corollary}[section]

\newtheorem{rem}{Remark}[section]
\newtheorem{exa}{Example}[section]

\title[A unified treatment of degenerate nonlocal elliptic problems]{A unified treatment of degenerate nonlocal elliptic problems}

\author{L. Gasi\'nski}
\address[L. Gasi\'nski]{Department of Mathematics, University of the National Education Commission, ul. Podchor\k{a}\.{z}ych 2, 30-084 Krak\'ow, Poland}
\email{leszek.gasinski@uken.krakow.pl}

\author{H. Ramos Quoirin}
\address[H. Ramos Quoirin]{CIEM-Conicet, Universidad Nacional de C\'ordoba, (5000) C\'ordoba, Argentina}
\email{humbertorq@gmail.com}

\author{J.R. Santos J\'unior}
\address[J.R. Santos J\'unior]{Instituto de Ci\^encias Exatas e Naturais, Universidade Federal do Par\'a, Avenida Augusto Corr\^ea 01, 66075-110, Bel\'em, PA, Brazil}
\email{joaojunior@ufpa.br}

\author{K. Silva}
\address[K. Silva]{Instituto de Matem\'atica e Estat\'istica, Universidade Federal de Goi\'as, Goi\^ania, GO 74001-970, Brazil}
\email{kayesilva@ufg.br}

\thanks{J.R. Santos J\'unior has been supported by CNPq/Brazil-Grant 304340/2025-1 and K. Silva has been supported by CNPq/Brazil-Grant 300849/2025-7}

\makeatletter
\newcommand{\mylabel}[2]{#2\def\@currentlabel{#2}\label{#1}}
\makeatother

\begin{document}

    \begin{abstract}
We develop a unified framework for a broad class of nonlocal elliptic problems, encompassing a wide spectrum of nonlocal terms, including the classical Kirchhoff and Carrier-type equations as particular cases, and nonlinearities having sublinear or asymptotically linear growth. By combining the study of a suitable auxiliary problem and fixed-point techniques with careful parameter analysis, we establish existence, non-existence, and multiplicity results for positive solutions. Our method reveals sharp parameter thresholds and provides a comprehensive description of the solution set. Finally, for powerlike nonlinearities (including superlinear and singular ones) we provide a more direct approach, based on homogeneity. 
        \end{abstract}

\maketitle

\begin{center}
    \begin{minipage}{12cm}
        \tableofcontents
    \end{minipage}
\end{center}

 \bigskip


\smallskip


\section{Introduction}

The mathematical modeling of transverse vibrations in elastic strings and membranes has a rich history, originating with the classical D'Alembert wave equation. However, a more realistic description of these phenomena was independently introduced by Kirchhoff \cite{Kirchhoff1883} in 1883 and later by Carrier \cite{Carrier1945} in 1945. Kirchhoff's model extends the classical theory by accounting for the changes in tension caused by the elongation of the string during vibration, a refinement that leads to the now-famous nonlocal term $a(\int_\Omega |\nabla u|^2 dx)\Delta u$. While groundbreaking, Kirchhoff's derivation is based on an approximation valid for small vibrations, resulting in a nonlocality that depends on the gradient's $L^2$-norm. In contrast, the model proposed by Carrier \cite{Carrier1945} involves a nonlocal coefficient that depends on the $L^2$-norm of the solution itself, $a(\int_\Omega u^2 dx)$, capturing a different, yet equally significant, physical mechanism. From a physical perspective, both models represent substantial improvements over the D'Alembert formulation, but mathematically, they present distinct challenges. A principal difference lies in the fact that Kirchhoff-type problems possess a natural variational structure, allowing solutions to be characterized as critical points of an associated energy functional. This variational framework has been extensively exploited over the past decades, leading to a vast literature on existence, multiplicity, and regularity under various assumptions on the nonlinearity and the function $a$ \cite{Alves2005, Anello2011, Li2012, Ma2003, Perera2006, Zhang2006}. Carrier-type problems, by contrast, lack this inherent variational principle and typically require the development of alternative, non-variational techniques, such as fixed-point arguments, bifurcation theory, or sub-supersolution methods \cite{Alves2015, Chipot1997, Chipot1992, Correa2011, Doo2016, Jin2020}. The study of both classes has been further diversified by considering a wide array of nonlinearities, boundary conditions, and degenerate coefficients. This rich interplay between physical motivation and mathematical challenge has also inspired the investigation of analogous nonlocal models for elastic beam equations, where biharmonic operators replace the Laplacian \cite{Arosio1998, Arosio1999, Arosio2001}.

Although the vast majority of models for Kirchhoff and Carrier-type equations considered nonlocal terms that were bounded below by a positive constant, thus ensuring the problem remained non-degenerate, a fundamental shift was initiated by Ambrosetti and Arcoya \cite{Ambrosetti2016, Ambrosetti2017}. For Kirchhoff-type problems, they began to investigate the mathematical consequences on the structure of the solution set when the Kirchhoff function vanishes at zero or at infinity. From a mathematical standpoint, these questions were particularly intriguing as they introduced significant technical challenges; the possibility of a vanishing coefficient rendered the problem degenerate, rendering many of the standard tools and techniques commonly employed for this class of problems inapplicable. Consequently, this class of degenerate problems demanded the development of distinct analytical mechanisms.

The first articles to uncover a precise relationship between the number of zeros of the nonlocal term, its points of degeneracy, and the number of solutions were \cite{Arcoya2021, Santos2018} for Kirchhoff problems (see also \cite{FS} for improvements and new results and also the forthcoming paper \cite{AFS} with even more generalizations) and \cite{Gasinski2019, Gasinski2020} for Carrier-type equations. While both works revealed this novel phenomenon, they did so through fundamentally distinct technical approaches. The former exploited the variational structure of the Kirchhoff problem, employing truncation techniques and variational methods, whereas the latter developed a non-variational approach centered on the study of a suitable auxiliary problem and fixed-point techniques, which was essential for tackling the Carrier-type nonlocality. Following these pioneering works, a growing number of papers have addressed degenerate Kirchhoff or Carrier problems using a variety of different approaches. Among these, we highlight the contributions of \cite{Candito2022, Candito2024}, who extended the analysis to the $p$-Laplacian and $p(x)$-Laplacian framework.

In the present paper, we preserve the core spirit of the approach developed in \cite{Gasinski2019, Gasinski2020} for degenerate Carrier problems. This strategy consists of first considering a suitable auxiliary problem, then, based on the fact that this auxiliary problem possesses a unique solution for each fixed parameter, defining an appropriate map, proving its continuity, and finally exploring the geometry of the functions involved to establish the existence of fixed points, which in turn yields solutions to the original problem. However, while maintaining this practical essence, we introduce several refinements and new perspectives that allow us to extend the technique to a much more general context. These improvements yield some advantages: 1) our approach is applicable not only to Carrier and Kirchhoff-type problems but to a far broader class of nonlocal problems; 2) the strategy developed herein also adapts to nonlinearities that could not be handled by the original method of \cite{Gasinski2019, Gasinski2020}, as we remove the assumption that the nonlinearity vanishes at some positive point. This allows us to study much more general nonlinearities, including those exhibiting sublinear and asymptotically linear growth. We also show that for superlinear power-type nonlinearities (where the auxiliary problem lacks uniqueness) we can still obtain existence, non-existence, and multiplicity results. While some attempts in this direction have already appeared in the literature using different auxiliary problems combined with global bifurcation theory \cite{Cintra2020}, their approach is specifically tailored to Kirchhoff problems, whereas our framework unifies both Kirchhoff and Carrier-type equations under a single, comprehensive treatment.

Previous attempts to develop a general framework for a broad class of problems involving degenerate nonlocal terms of Kirchhoff and Carrier type can be found in the work \cite{Gasinski2022}. In that paper, the authors follow a different strategy by considering an equivalent formulation in which the nonlocal term appears on the right-hand side of the equation, in the denominator. Such a reformulation effectively leads to a problem whose nonlinearity displays a singular behavior when viewed through its dependence on the nonlocal term. Furthermore, the analysis in that work relies on additional restrictive assumptions on the nonlinearity, notably the existence of a positive point where the function vanishes, which allows the use of truncation arguments. As a consequence, their approach does not cover several important classes of nonlinearities treated in the present paper, including sublinear and asymptotically linear nonlinearities.

\subsection{General Nonlocal Equation and Background}
Consider the following equation
\begin{equation}
\tag{$\mathcal{P}_g$}\label{Pg}
\begin{cases}
- a\!\left(g(u) \right)\Delta u = \lambda f(u) & \text{in } \Omega, \\[6pt]
u > 0 & \text{in } \Omega, \\[6pt]
u = 0 & \text{on } \partial\Omega,
\end{cases}
\end{equation}
where $\Omega\subset\mathbb{R}^N$ is a bounded smooth domain,  $\lambda>0$, $a,f:[0,\infty) \rightarrow [0,\infty)$ are continuous, and $g:C^1_0(\overline{\Omega})\to [0,\infty)$ is a nonlocal continuous functional. The main examples of $g$ we are interested in are given by
\begin{equation*}
g(u)=\int_\Omega |u|^\gamma dx \ \ \mbox{and}\ \ \ g(u)=\int_\Omega |\nabla u|^\gamma dx,  \ \gamma\in [1,\infty).
\end{equation*}
However, as will be clear later on, our assumptions on $g$ cover a more general class of nonlocal terms.
We shall assume that $a$ has $k$ zeros:

\begin{enumerate}
\item[(\mylabel{a1}{$a_1$})] There exist $0:=t_0<t_1<t_2<\ldots<t_k$ such that $a(t_i)=0$, $i=1,2,\ldots,k$
and $a(t)>0$ for $t\in(t_i,t_{i+1})$, $i=0,\ldots,k-1$.
\end{enumerate}
As for $f$, we assume it to be sublinear or asymptotically linear. In addition, some monotonicity and positivity are also required:
\begin{enumerate}
\item[(\mylabel{f1}{$f_1$})] $f(0)=0$,
$f(t)>0$  \text{for all } $t>0$, and $\psi(t):=\frac{f(t)}{t}$  \text{is strictly decreasing in }$(0,\infty)$. In particular
$0 \leq \theta<\beta \leq \infty$,
where
\begin{equation*}
\theta:=\lim_{t\to\infty}\frac{f(t)}{t}\quad \mbox{and} \quad \beta:=\lim_{t\to 0^+}\frac{f(t)}{t}.
\end{equation*}
\item[(\mylabel{f2}{$f_2$})] If $\theta=0$ then $\displaystyle \lim_{t\to\infty}f(t)/t^{p-1}=c_0$, for some $c_0>0$ and  $1<p<2$.
\end{enumerate}

From \eqref{a1} and \eqref{f1} it is clear that there are no solutions satisfying $g(u)=t_i$, so we shall be concerned with solutions of \eqref{Pg} satisfying $g(u)\neq t_i$. 
Given $0\le t'<t''$ we shall say that $u$ is a solution of \eqref{Pg} in the interval $(t',t'')$ if $u$ is a solution to \eqref{Pg} and $t'<g(u)<t''$.\\

We shall approach \eqref{Pg} by looking for solutions of the problem
\begin{equation*}
\begin{cases}
-a(\alpha)\Delta u = \lambda f(u) & \text{in }\Omega,\\
u>0 & \text{in }\Omega,\\
u=0 & \text{on }\partial\Omega.
\end{cases}
\end{equation*}
that additionally satisfy $g(u)=\alpha$, for some $\alpha>0$ (note that weak solutions of  this problem belong to $C^{1, \alpha}(\overline{\Omega})$ by standard elliptic regularity, so that $g(u)$ is well-defined).
Thus the following auxiliary problem will play an important role in this work:
\begin{equation*}
\begin{cases}
-\Delta w = s\, f(w) & \text{in }\Omega,\\
w>0 & \text{in }\Omega,\\
w=0 & \text{on }\partial\Omega.
\end{cases}
\end{equation*}

Let $s\in(\lambda_1/\beta,\lambda_1/\theta)$, where $\lambda_1$ is the first eigenvalue of the Dirichlet Laplacian, and we agree that $\lambda_1/0=\infty$ and $\lambda_1/\infty=0$.
By \cite{Brezis1986} this problem has a unique solution $w_s$.
We set $$Q(s):=g(w_s), \quad \mbox{for} \quad s \in (\lambda_1/\beta,\lambda_1/\theta),$$
 and we assume that $g(0)=0$,  and one of the following conditions hold:

\begin{enumerate}
\item[(\mylabel{g1}{$g_1$})]  $g(u_n)\to\infty$ as $|u_n(x)|\to\infty$ in a positive measure subset of $\Omega$.

\item[(\mylabel{g2}{$g_2$})]  $g(u_n)\to\infty$  as $|\nabla u_n(x)|\to\infty$ in a positive measure subset of $\Omega$.


\end{enumerate}



\begin{rem}\strut
\begin{enumerate}
\item Note that \eqref{g1} is clearly satisfied by $g(u)=\int_\Omega |u|^\gamma dx$, whereas 
\eqref{g2} is satisfied by $g(u)=\int_\Omega |\nabla u|^\gamma dx$, for any   $\gamma\in [1,\infty)$. Further examples satisfying \eqref{g1} or \eqref{g2} will be given in Example \ref{exam} below.
\item Under \eqref{g1} or \eqref{g2} we shall see that the map $Q$ is continuous, and satisfies $Q(s) \to 0$ as $s \to \lambda_1/\beta$, and $Q(s) \to \infty$ as $s \to \lambda_1/\theta$, cf. Lemma \ref{l1} below. The latter condition is in fact enough for our purposes, and it is more general than assuming \eqref{g1} or \eqref{g2}. Indeed, if $\Omega'$ is a positive measure subset of $\Omega$ with $|\Omega \setminus \Omega'|>0$ and
$g(u)=\int_{\Omega'} |u|^\gamma$ then neither \eqref{g1} nor \eqref{g2} is satisfied. However, for such $g$ we still have  $Q(s) \to \infty$ as $s \to \lambda_1/\theta$, cf. Lemma \ref{l1}(iii) below.
\end{enumerate}	

\end{rem}

Given  $i=0,1,\ldots,k-1,$ we set
\begin{equation*}
    A_i:=\max\{a(\alpha): \alpha\in [t_i,t_{i+1}]\}.
\end{equation*}
From \eqref{a1} it is clear that $A_i>0$.

Our first and most general result reads as follows:

\begin{theor}\label{thm-1} Suppose \eqref{a1}, \eqref{f1} and either \eqref{g1} or \eqref{g2} and \eqref{f2}. Then  for any $i=0,1,\ldots,k-1$ there exists $0<\lambda_{0,i}<\widetilde\lambda_{0,i}<(A_i\lambda_1)/\theta$ such that the following assertions hold:
\begin{enumerate}
    \item[$(i)$] If $0<\lambda<\lambda_{0,i}$ then \eqref{Pg} has at least two solutions $u_{1,i},u_{2,i}$ in $(t_i,t_{i+1})$.
\item[$(ii)$] If $\lambda>\widetilde \lambda_{0,i}$ then \eqref{Pg} has no solution in $(t_i,t_{i+1})$.
\end{enumerate}
In particular:
\begin{enumerate}
\item[$(iii)$] If $\lambda<\displaystyle \min_{i=0,1\ldots,k-1}\lambda_{0,i}$, then  \eqref{Pg} has at least $2k$ solutions $u_{1,i},u_{2,i}$ in $(t_i,t_{i+1})$, $i=0,1,\ldots,k-1$.
   \item[$(iv)$]  if $\lambda>\displaystyle  \max_{i=0,1\ldots,k-1}\widetilde \lambda_{0,i}$, then for any $i=0,1,\ldots,k-1$ there is no solution of \eqref{Pg} in $(t_i,t_{i+1})$.
\end{enumerate}
If, in addition,  $Q$ is increasing, then $\lambda_{0,i}=\widetilde\lambda_{0,i}$ and for $\lambda=\lambda_{0,i}$ there exists at least one solution $u_i$ of \eqref{Pg}  in $(t_i,t_{i+1})$.
\end{theor}

\begin{rem}
Note that $Q$ depends only on $f$ and $g$, i.e. once $f$ and $g$ are fixed, $Q$ is fixed too. In some cases it has a computable expression (see \eqref{deel}) but, in general, we can only bound it from below and above (see Lemma \ref{lbounds} and Corollary \ref{c1}). We shall see that $Q$ is continuous under \eqref{f1} and \eqref{g1}. Thus, if in addition $Q$ is increasing then it is a homeomorphism, and we have
\begin{equation}\label{defl0}
\lambda_{0,i}=\displaystyle \max_{\alpha\in [t_i,t_{i+1}]}a(\alpha)Q^{-1}(\alpha)
\end{equation}
for $i=0,1,\ldots,k-1$.
\end{rem}

We are mainly interested in two particular choices of $g$, namely:
\begin{equation*}
g(u)=\int_\Omega |u|^{\gamma_1} dx \ \ \mbox{and}\ \ \ g(u)=\int_\Omega |\nabla u|^{\gamma_2} dx,  \quad \gamma_1,\gamma_2 \in [1,\infty).
\end{equation*}
As already observed, such $g$ satisfy \eqref{g1} or \eqref{g2}. Furthermore,
we shall see that the corresponding $Q$ is increasing (assuming $\gamma_2=2$), cf. Lemma \ref{hlemma}.
Theorem \ref{thm-1} yields then following result:
\begin{theor}\label{thm1} Suppose \eqref{a1}, \eqref{f1} and either $g(u)=\int_\Omega |u|^{\gamma} dx$, $\gamma \in [1,\infty)$ or $g(u)=\int_\Omega |\nabla u|^{2} dx$. Then $Q$ is increasing and the conclusions of Theorem \ref{thm-1} hold with $\lambda_{0,i}$  given by \eqref{defl0}.
\end{theor}

Theorem \ref{thm-1} can be applied in many situations, as the following examples of $f$ and $g$ (with the corresponding values of $\beta$ and $\theta$) satisfy \eqref{f1} and \eqref{g1}, respectively:
\begin{exa} \strut
\begin{enumerate}
    \item $f(t)=t^{p-1}$, $p\in(1,2)$, $\beta=\infty$ and $\theta=0$.
 
    \item $f(t)=\frac{\beta_0 t}{1+t}$ with $\beta_0>0$,  $\beta=\beta_0$ and $\theta =0$.
       \item $f(t)=\theta_0 t+\sqrt{t}$ with $\theta_0>0$,  $\beta=\infty$ and $\theta =\theta_0$.
       \item $f(t)=t\left(\theta_0+\frac{\beta_0-\theta_0}{1+t}\right)$ or $f(t)=t\left[\theta_0+(\beta_0-\theta_0)\left(1-\frac{2}{\pi}\arctan t\right)\right]$ with  $0<\theta_0<\beta_0$,  $\beta=\beta_0$ and $\theta =\theta_0$.
       
\end{enumerate}
\end{exa}

\begin{exa}\label{exam}\strut

\begin{enumerate}
	
\item \textbf{Functionals depending on $L^p, L^\infty, C^1_0, W_0^{1,p}$-norms:} Let $p\geq 1$ and $\|u\|$ be any one of the following norms: $\|u\|_{L^p(\Omega)}, \|u\|_{L^\infty(\Omega)}, \|\nabla u\|_{L^p(\Omega)}$ or $\|\nabla u\|_{L^\infty(\Omega)}$.
A very natural class of $g$ is obtained by composing $\|.\|$ with a suitable continuous function. 
Let $\phi:[0,+\infty)\to[0,+\infty)$ be continuous, with $\phi(0)=0$ and $\phi(t)\to\infty$ as $t\to\infty$. Then $g(u)=\phi(\|u\|)$ satisfies the required properties.

\medskip
Typical examples include polynomial, logarithmic, exponential and super-exponential growth:

\begin{itemize}
\item[$(i)$] $g(u)=\|u\|^q, \ q>0$;
\item[$(ii)$] $g(u)=\log(1+\|u\|)$;
\item[$(iii)$] $g(u)=\log\bigl(1+\log(1+\|u\|)\bigr)$;
\item[$(iv)$] $g(u)=e^{\|u\|^q}-1, \ q>0$;
\item[$(v)$] $g(u)=\exp(\exp(\|u\|))-e$.\\
\end{itemize}

\item \textbf{General integral functionals:} Another large class is given by nonlinear integral functionals of the form

\[
g(u)=\int_\Omega \phi(|u(x)|)\,dx, \ \text{or} \ \ g(u)=\int_\Omega \phi(|\nabla u(x)|)\,dx,
\]
where $\phi$ is as previously. Examples include:

\begin{itemize}
\item[$(i)$] $g(u)=\int_\Omega \log(1+|\nabla u|)\,dx$;
\item[$(ii)$] $g(u)=\int_{\Omega}|u|^q\log(1+|u|)\,dx, \ q\geq 0$;
\item[$(iii)$] $g(u)=\int_{\Omega}(e^{|u|^q}-1) dx, \ q>0$;
\item[$(iv)$] $g(u)=\int_\Omega e^{|\nabla u|}\,dx-|\Omega|$.\\
\end{itemize}

\end{enumerate}

Let us check that all the previous examples satisfy \eqref{g1} or \eqref{g2}: let $(u_n) \subset C^1_0(\overline{\Omega})$.
If $|u_n(x)| \to +\infty$ for $x \in \Omega_1$, with  $|\Omega_1|>0$,
then it follows immediately from Fatou's Lemma that $\|u_n\|_{L^1(\Omega)} \to +\infty$. The continuous embeddings $L^p(\Omega), L^{\infty}(\Omega) \subset L^1(\Omega)$
yield that
$\|u_n\|_{L^p(\Omega)}\to+\infty$
for every  $p \ge1$, and
$\|u_n\|_{L^\infty(\Omega)}\to+\infty$,
whereas Sobolev and Holder inequalities provide us with $\|\nabla u_n\|_{L^p(\Omega)} \to+\infty$ and $\|\nabla u_n\|_{L^\infty(\Omega)} \to+\infty$.

If now $|\nabla u_n(x)|\to+\infty$
for $x \in \Omega_2$, with  $|\Omega_2|>0$,
then it is immediate to conclude that all the examples of general integral functionals involving $|\nabla u|$ satisfy \eqref{g2}.

\end{exa}

If $a$ has an oscillatory behavior then a high number of solutions may arise. To simplify we assume that
$a$ has only two zeros:

\begin{theor}\label{thm6} Assume the conditions of Theorem \ref{thm-1} with $k=1$,  and $Q$ is increasing. Let $H(\alpha):=a(\alpha)Q^{-1}(\alpha)$, $M_1$ be the set of local maximizers of $H$ and $M_2$ be the set of local minimizers of $H$. Assume that both sets are finite and $m:=\sup\{H(\alpha):\alpha\in M_2\}<M:=\inf\{H(\alpha):\alpha\in M_1\}$. If $j$ denotes de cardinality of $M_1$, then problem \eqref{Pg} has at least $2j$ solutions in $(t_0,t_1)$ for all $\lambda\in(m,M)$.
	
\end{theor}

Note that since $Q^{-1}:(0,\infty)\to (\lambda_1/\beta,\lambda_1/\theta)$ is an increasing, continuous and positive function, we can always consider, in Theorem \ref{thm6}, the change of variables $a(\alpha)=b(\alpha)/Q^{-1}(\alpha)$.

\vskip.3cm
We can also prove existence of infinitely many solutions under the following condition:

\begin{enumerate}
	\item[(\mylabel{a2}{$a_2$})] There exist $0:=t_0<t_1<t_2<\ldots$ such that $a(t_i)=0$, $i\in\mathbb{N}$
	and $a(t)>0$ for $t\in(t_i,t_{i+1})$, $i\in\mathbb{N}\cup\{0\}$.
\end{enumerate}

\begin{cor}\label{thm2} Suppose \eqref{a2}, \eqref{f1} and either \eqref{g1} or \eqref{g2} and \eqref{f2}. In addition, assume that $Q$ is increasing.
	\begin{enumerate}
		\item[$(i)$] If $\bar\lambda_{0}:=\displaystyle \min_{i\in\mathbb{N}\cup\{0\}}\max_{\alpha\in [t_i,t_{i+1}]}(a(\alpha)Q^{-1}(\alpha))>0$ and $0<\lambda<\bar\lambda_{0}$ then for any $i\in\mathbb{N}\cup\{0\}$ the problem \eqref{Pg} has a pair of solutions $u_{1,i},u_{2,i}$ in $(t_i,t_{i+1})$. In particular \eqref{Pg} has two sequences of solutions.
		\item[$(ii)$] If $\tilde\lambda_{0}:=\displaystyle \max_{i\in\mathbb{N}\cup\{0\}}\max_{\alpha\in [t_i,t_{i+1}]}(a(\alpha)Q^{-1}(\alpha))<\infty$ and $\lambda>\tilde\lambda_{0}$ then problem \eqref{Pg} has no solution in $(t_i,t_{i+1})$.
		
	\end{enumerate}
\end{cor}

\medskip
\subsection{The Powerlike Case}\label{pow}

Let  $f(t)=t^{p-1}$, $t> 0$, where  $p\neq 2$. In particular, we include the ranges $p>2$ (the superlinear case) and $-\infty<p<1$ (the singular case), which do not satisfy \eqref{f1}.
For such $f$ we can treat \eqref{Pg} by a simple procedure relying on homogeneity.
Indeed, assume that $v$ is a positive solution of the problem
\begin{equation}\label{le}
-\Delta v = v^{p-1} \quad \text{in }\Omega,
\qquad
v=0 \quad \text{on }\partial\Omega.
\end{equation}
and $C_v:=g(v)$. Let us assume that $g$ is a homogeneous functional, i.e.
$g(tu)=t^{\gamma}g(u)$ for some $\gamma>0$ and any $t>0$, $u \in C_0^1(\overline{\Omega})$. 
Then it is straightforward to see that
$u=sv$ (with $s>0$) solves \eqref{Pg} if, and only if,
\begin{equation*}
\lambda=\left(\frac{C_v}{\alpha}\right)^{\frac{p-2}{\gamma}} a(\alpha), \quad \alpha=C_v s^{\gamma}.
\end{equation*}
Note also that in this case $g(u)=g(sv)=C_v s^{\gamma}=\alpha$.
Therefore, setting
\begin{equation}\label{deel}
\lambda_{0,i}:=C_v^{\frac{p-2}{\gamma}}\displaystyle \max_{\alpha\in[t_i,t_{i+1}]}
a(\alpha)\alpha^{\frac{2-p}{\gamma}}
\end{equation}
 for $i=0,1,\ldots,k-1$, we obtain the following results:

\begin{theor}\label{thm3} Under the previous conditions on $f$, $g$, assume that $p<2$ and $a$ satisfies \eqref{a1}.
Then the following assertions hold:

\begin{enumerate}
\item[$(i)$] If $0<\lambda<\lambda_{0,i}$ then \eqref{Pg} has at least two solutions $u_{1,i},u_{2,i}$ in $(t_i,t_{i+1})$.

\item[$(ii)$] If $\lambda=\lambda_{0,i}$ then  \eqref{Pg} has at least one solution $u_i$ in $(t_i,t_{i+1})$.

\item[$(iii)$] If $\lambda>\lambda_{0,i}$ then problem \eqref{Pg} has no solution in $(t_i,t_{i+1})$.
\end{enumerate}
If moreover either $a$ or $\alpha \mapsto \alpha^{\frac{2-p}{\gamma}} a(\alpha)$ is strictly concave in $(t_i,t_{i+1})$ then $(i)$ and $(ii)$ hold with 'at least' replaced by 'exactly'.
\end{theor}

The exactness results in the previous theorem follow from the existence and uniqueness of positive solution of \eqref{le} in the sublinear and singular cases, see e.g. \cite{Brezis1986} for $1<p<2$ and \cite{OP} for $-\infty<p<1$. Whereas existence also holds for $2<p<2^*$, uniqueness fails in general  as this problem may have multiple positive solutions (which cannot be multiples of each other), at least for some choices of $\Omega$, cf. \cite{BL} and references therein. In such case we obtain a high number of positive solutions of \eqref{Pg} for $\lambda>0$ small enough:

\begin{theor}\label{thm3'} Under the previous conditions on $f$ and $g$, assume that $2<p<2^*$ and $a$ satisfies \eqref{a1}. Let $v$ be a positive solution of \eqref{le} and  assume that
$$
\mu_0:=\lim_{\alpha\to 0^+}a(\alpha)\alpha^{\frac{2-p}{\gamma}}=0.
$$
Then the following assertions hold:

\begin{enumerate}
\item[$(i)$] If $0<\lambda<\lambda_{0,i}$ then \eqref{Pg} has at least two solutions $u_{1,i}=s_{1,i}v$ and $u_{2,i}=s_{2,i}v$ in $(t_i,t_{i+1})$.

\item[$(ii)$] If $\lambda=\lambda_{0,i}$ then  \eqref{Pg} has at least one solution $u_i=s_i v$ in $(t_i,t_{i+1})$.

\item[$(iii)$] If $\lambda>\lambda_{0,i}$ then problem \eqref{Pg} has no solution of the form $u=sv$ in $(t_i,t_{i+1})$.
\end{enumerate}
If moreover either $a$ is strictly concave in $(t_i,t_{i+1})$ and $\gamma+2<p<2^*$ or $\alpha \mapsto \alpha^{\frac{2-p}{\gamma}} a(\alpha)$ is strictly concave, then $(i)$ and $(ii)$ hold with 'at least' replaced by 'exactly'.
\end{theor}
Let us stress that exactness in Theorem \ref{thm3'} is meant among solutions in $(t_i,t_{i+1})$ of the form $u=sv$.
\begin{rem}\label{rm1}
Let $i=0$ and $2<p<2^*$. If $\mu_0>0$ then we have the following alternative in Theorem \ref{thm3'} (we write $\lambda_0=\lambda_{0,0}$):
\begin{enumerate}
\item If $\mu_0<\lambda_0$ then \eqref{Pg} has at least one solution $u=sv$ in $(0,t_1)$ for  $0<\lambda\le\mu_0$ and $\lambda=\lambda_0$. If $\lambda\in (\mu_0,\lambda_0)$, then it has at least two solutions $u=s_1v$ and $u=s_2v$ in $(0,t_1)$.
\item If $\mu_0\ge\lambda_0$ (including the case $\mu_0=\infty$), then problem \eqref{Pg} has at least one solution $u=sv$ in $(0,t_1)$ for  $0<\lambda<\mu_0$.
\end{enumerate}
\end{rem}

\noindent

Here are some examples:

\begin{exa}\label{exa1} Let $a(\alpha)=|\sin(\alpha)|$, $\alpha\ge 0$, and $f(t)=t^{p-1}$, $t> 0$, where $p<2^*$, $p\neq 2$. If $0<\lambda< \displaystyle\left(\frac{\pi}{2C_v}\right)^\frac{2-p}{\gamma}$ then  \eqref{Pg} has at least two solutions $u_{1,i},u_{2,i}$ in $(i\pi,(i+1)\pi)$ for any $i\in\mathbb{N}\cup\{0\}$ $(i\in \mathbb{N} \ \text{if} \ p>2)$.
Moreover, the previous statement holds with 'at least' replaced by 'exactly'  if either $1<p<2$ or $\gamma+2<p<2^*$ (in the latter case exactness holds among multiples of $v$)
\end{exa}

A result similar to the latter one was obtained for $2<p<2^*$ and $g(u)=\int_\Omega |\nabla u|^2 dx$ in \cite[Theorem 5.4]{FS}.
\vskip.3cm

\begin{exa}\label{exa2} Let $\displaystyle a(\alpha)=|\sin(\alpha)|\alpha^{\frac{p-2}{\gamma}}$, $\alpha\ge 0$, and $f(t)=t^{p-1}$, $t> 0$, where $p<2^*$, $p\neq 2$.
\begin{enumerate}
    \item[$(i)$] If $0<\lambda< C_v^\frac{p-2}{\gamma}$ then for any $i\in \mathbb{N}\cup\{0\}$ the problem  \eqref{Pg} has at least two  solutions $u_{1,i},u_{2,i}$ such that $i\pi<g(u_{1,i}) <g(u_{2,i}) <(i+1)\pi$.
In particular \eqref{Pg} has at least two sequences of solutions.
   \item[$(ii)$] If $\lambda=C_v^\frac{p-2}{\gamma}$ then for any $i\in \mathbb{N}$ the problem \eqref{Pg} has at least one solution $u_i$ such that
$g(u_i) =\frac{\pi}{2}+i\pi$.
In particular \eqref{Pg} has at least a sequence of solutions.
  \item[$(iii)$] If $\lambda>C_v^\frac{p-2}{\gamma}$ then for any $i\in \mathbb{N}\cup\{0\}$ there is no solution of \eqref{Pg} of the form $u=sv$ satisfying
$i\pi<g(u)<(i+1)\pi$.

\end{enumerate}
Moreover, $(i)$ and $(ii)$ hold with 'at least' replaced by 'exactly'  (if $p>2$ then exactness holds among multiples of $v$).
   
\end{exa}
\medskip

Finally, we note that the method described in the Introduction in the case $1<p<2$ provides the same value of $\lambda_{0,i}$. Indeed, since $g$ is $\gamma$-homogeneous we have
\[
Q(s)=g(w_s)
=C_v\,s^{\frac{\gamma}{2-p}}.
\]

In particular,
\[
Q^{-1}(\alpha)=\left(\frac{\alpha}{C_v}\right)^{\frac{2-p}{\gamma}}, \quad \mbox{i.e.} \quad a(\alpha)\,Q^{-1}(\alpha)=a(\alpha)\left(\frac{\alpha}{C_v}\right)^{\frac{2-p}{\gamma}}
\qquad \mbox{for } \alpha>0.
\]

\medskip
\textbf{Outline of the paper:}
This paper is organized as follows. In Section \ref{secauxiliary}, we introduce and study the auxiliary problem, establishing the existence and uniqueness of its solution and the key properties of the map $Q$. In Section \ref{secfg}, we develop the fixed-point framework and establish the existence of fixed points under general conditions. Section \ref{sqincreasing} refines these results for the case where $Q$ is increasing, leading to sharp parameter thresholds. In Section 4 we prove our main theorems, which are then applied to the classical Kirchhoff and Carrier equations, and to problems with power-type nonlinearities. Finally, in Section 5 we show that assuming a bit more regularity on $f$ allows us to deal with another class of $g$. 

\section{The Auxiliary Problem}\label{secauxiliary}

In this section we study the auxiliary problem
\begin{equation}\label{ws}
\begin{cases}
-\Delta w = s\, f(w) & \text{in }\Omega,\\
w>0 & \text{in }\Omega,\\
w=0 & \text{on }\partial\Omega,
\end{cases}
\end{equation}
where $s>0$ is fixed and $f$ satisfies \eqref{f1}.

We extend $f$ to $\mathbb R$ by
$f(t)=0$ for $t<0$, and we set
$$
F(t):=\int_0^t f(\tau)\,d\tau \ \text{for }t\in\mathbb R,
$$
so that $F(t)=0$ for $t\le 0$. Let us set, for $s>0$,
\[
\Phi_s(u):=\frac12\int_\Omega |\nabla u|^2\,dx \;-\; s\int_\Omega F(u)\,dx,
\qquad u\in H_0^1(\Omega).
\]

\begin{lem}\label{lem:unified-minimizer}

Let $c_s:=\inf \Phi_s$. The following statements hold:
\begin{enumerate}
\item[(i)] $c_s$ is attained for every $s<\lambda_1/\theta$.
\item[(ii)] If $K\subset (\lambda_1/\beta,\lambda_1/\theta)$ is a compact set, then there exists $c<0$ such that $c_s<c$ for any $s \in K$. In particular,
every global minimizer $w_s$ is nontrivial and satisfies $w_s>0$.
\item[(iii)] Problem \eqref{ws} has a solution if, and only if, $s\in (\lambda_1/\beta,\lambda_1/\theta)$. Moreover, in this case the solution of \eqref{ws} is unique, and it is given by $w_s$. Furthermore $w_s\in C^1(\overline{\Omega})$.
\end{enumerate}
\end{lem}

\begin{proof}

Fix $\varepsilon>0$. By the definition of $\theta$ there exists $T_\varepsilon>0$ such that
$f(t)\le (\theta+\varepsilon)t$ for all $t\ge T_\varepsilon$.  Setting
$M_\varepsilon:=\max_{[0,T_\varepsilon]} f <\infty$, we have for every $t\ge 0$,
\begin{equation}\label{eq:growth-f}
f(t)\le (\theta+\varepsilon) t+M_\varepsilon.
\end{equation}
Integrating \eqref{eq:growth-f} from $0$ to $t\ge 0$ gives
\begin{equation}\label{eq:growth-F-pos}
F(t)\le \frac{\theta+\varepsilon}{2}t^2+M_\varepsilon t \qquad \forall\,t\ge 0.
\end{equation}
Since $F(t)=0$ for $t\le 0$, it follows that
\begin{equation}\label{eq:growth-F-all}
0\le F(t)\le  \frac{\theta+\varepsilon}{2}t^2+M_\varepsilon |t|
\qquad \forall\,t\in\mathbb R.
\end{equation}
Using Poincar\'e's and Holder's inequalities, we obtain
\[
\Phi_s(u)\ge \frac12\int_\Omega |\nabla u|^2\,dx
- s\Big( \frac{\theta+\varepsilon}{2\lambda_1}\int_\Omega |\nabla u|^2\,dx+CM_\varepsilon\|u\|\Big)
=
\frac12\Big(1-\frac{s(\theta+\varepsilon)}{\lambda_1}\Big)\|u\|_{H_0^1(\Omega)}^2 - sCM_\varepsilon\|u\|.
\]
Hence, for $s<\lambda_1/\theta$ we can choose $\varepsilon>0$ so small that $s(\theta+\varepsilon)<\lambda_1$.
Then $\Phi_s$ is coercive and weakly lower semicontinuous, so it attains its global minimum
at some $w\in H_0^1(\Omega)$.

\medskip

Let $\varphi_1>0$ be the $L^2$-normalized first eigenfunction. Since $f(t)/t\to\beta$ as $t\to0^+$, we have
\[
\lim_{t\to0^+}\frac{F(t)}{t^2}=\lim_{t\to0^+}\frac{f(t)}{2t}=\frac{\beta}{2}
\quad\text{(possibly $+\infty$ if $\beta=\infty$)}.
\]
Hence for $t>0$ small,
\[
\int_\Omega F(t\varphi_1)\,dx = \frac{\beta}{2}t^2\int_\Omega \varphi_1^2\,dx + o(t^2)
= \frac{\beta}{2}t^2 + o(t^2),
\]
and
\[
\Phi_s(t\varphi_1)=\frac12 t^2\int_\Omega |\nabla\varphi_1|^2\,dx - s\int_\Omega F(t\varphi_1)\,dx
=\frac12 t^2\lambda_1 - s\Big(\frac{\beta}{2}t^2 + o(t^2)\Big).
\]
Thus for $s>\lambda_1/\beta$ we find that
$\Phi_s(t\varphi_1)<c_s<0$ for $t>0$ small enough. It is also clear that if $s\in K\subset (\lambda_1/\beta,\lambda_1/\theta)$, then $c_s<c<0$ for some constant $c$. This proves $(i)$ and $(ii)$. Finally, $(iii)$ follows from \cite{Brezis1986} and standard regularity theory.

\end{proof}

From now on, we assume \eqref{g1}, and we recall that
\[
Q:(\lambda_1/\beta,\lambda_1/\theta)\to(0,\infty),
\qquad
Q(s):=g(w_s),
\]
where $w_s$ is given by Lemma \ref{lem:unified-minimizer}.  Furthermore, $Q$ has the following properties:

\begin{lem}\label{l1}\strut
\begin{enumerate}

\item[$(i)$]  If $\lambda_1/\beta<s_1<s_2<\lambda_1/\theta$, then $w_{s_1}\leq w_{s_2}$ and $w_{s_1}\not \equiv w_{s_2}$ in $\Omega$. 

\item[$(ii)$]  If $s_n\to s\in (\lambda_1/\beta,\lambda_1/\theta)$, then $w_{s_n}\to w_s$ in $C^1(\overline{\Omega})$.
Consequently, $Q$ is continuous.

\item[$(iii)$] $w_s \to 0$ in $C^1(\overline{\Omega})$ as $s \to (\lambda_1/\beta)^+$, and $\displaystyle \lim_{s\to (\lambda_1/\theta)^-}{w_s(x)}=\infty$, for all $x\in\Omega$.

\item[$(iv)$] The following convergence holds (if $\theta=0$ we further assume $\lim_{t\to\infty}f(t)/t^{p-1}=c_0$, for some $c_0>0$ and  $1<p<2$): $|\nabla w_s(x)|\to\infty$ as $s\to (\lambda_1/\theta)^-$, in some positive measure subset of $\Omega$.

\end{enumerate}
If, in addition, $Q$ is  increasing then it is a homeomorphism from $(\lambda_1/\beta,\lambda_1/\theta)$ onto $(0,\infty)$, and thus it
admits an inverse $Q^{-1}:(0,\infty)\to(\lambda_1/\beta,\lambda_1/\theta)$.
\end{lem}

\begin{proof} $(i)$:
Let $\lambda_1/\beta<s_1<s_2<\lambda_1/\theta$ and set $w_j=w_{s_j}$. Then
\[
-\Delta w_2=s_2 f(w_2)\ge s_1 f(w_2),
\]
so $w_2$ is a supersolution for the problem with parameter $s_1$. By comparison, $w_1\le w_2$ in $\Omega$.
By uniqueness, $w_1\not\equiv w_2$, hence $w_1<w_2$ in some positive measure subset of $\Omega$.

\medskip
$(ii)$:
Let $s_n\to s\in(\lambda_1/\beta,\lambda_1/\theta)$ and set $w_n:=w_{s_n}$.  Since $f(t)\le \beta t$ for all $t>0$, we have
\[
-\Delta w_n = s_n f(w_n) \le s_n \beta w_n.
\]
By $(i)$ we can assume, without loss of generality, that $w_1\ge w_n$, $n\in\mathbb{N}$, so $w_n$ is bounded in $L^\infty(\Omega)$ by Lemma \ref{lem:unified-minimizer} $(iii)$.  Standard elliptic regularity estimates imply that $\{w_n\}$ is bounded
in $C^{1,\alpha}(\overline\Omega)$ for some $\alpha\in(0,1)$.
Hence, up to a subsequence,
\[
w_n \to w \quad \text{in } C^1(\overline\Omega),\ \ \ \mbox{in particular}\ \ \ w_n \to w \quad \text{in } H_0^1(\overline\Omega).
\]
for some $w\ge 0$ with $w\in H_0^1(\Omega)$. Passing to the limit in the weak formulation yields that $w$ solves \eqref{ws} with parameter $s$.
By Lemma \ref{lem:unified-minimizer}, we have that $\Phi_s(w)<0$, so $w\neq 0$, and by uniqueness, $w=w_s$. By continuity of $g$, we conclude that $Q(s_n)\to Q(s)$.

\medskip
$(iii)$:
Let $s_n \to (\lambda_1/\beta)^+$ and set $w_n:=w_{s_n}$.
We claim that $\|w_n\|_{C^1(\Omega)} \to 0$. By arguing exactly as in the previous item (using this time Lemma \ref{l1}(i), we can assume that $w_1\ge w_n$, $n\in\mathbb{N}$), up to a subsequence, we get
\[
w_n \to w \quad \text{in } C^1(\overline\Omega),
\]
for some $w\ge 0$ with $C^1(\overline\Omega)$. Passing to the limit in the weak formulation gives
\[
-\Delta w = (\lambda_1/\beta) f(w)
\quad \text{in } \Omega,
\qquad
w=0 \text{ on } \partial\Omega.
\]
By Lemma~\ref{lem:unified-minimizer}$(iii)$, problem \eqref{ws}
has no positive solution for $s=\lambda_1/\beta$.
Thus $w\equiv 0$ and $\|w_n\|_{C^1(\Omega)}\to 0$ as claimed.

\medskip

Assume $s\to(\lambda_1/\theta)^-$. Let $\varphi_1$ be the first eigenfunction of $-\Delta$ normalized in $L^\infty(\Omega)$, so $0<\varphi_1\le 1$ and $-\Delta \varphi_1=\lambda_1 \varphi_1$.
Set $\eta_s:=\psi^{-1}(\lambda_1/s)$ and $z_s:=\eta_s \varphi_1$. Recall that $\psi:(0,\infty)\to (\theta,\beta)$ is a decreasing homeomorphism. Since $\psi$ is decreasing and $0<\varphi_1\le 1$, we have $\psi(\eta_s \varphi_1)\ge \psi(\eta_s)=\lambda_1/s$, hence
\[
-\Delta z_s=\lambda_1\eta_s \varphi_1 \le s f(\eta_s \varphi_1)=s f(z_s),
\]
so $z_s$ is a subsolution and thus $w_s\ge z_s$ by comparison. Because $\lim_{t\to\infty}\psi(t)=\theta$, we have $\psi^{-1}(y)\to\infty$ as $y\to \theta^+$. Since $\lambda_1/s\to \theta^+$ as $s\to(\lambda_1/\theta)^-$, we get $\eta_s\to\infty$ and the result follows.

\medskip

$(iv)$ 
We are going to consider two cases.

\noindent
{\bf Case 1: $\theta>0$.} 

Let $s<\lambda_1/\theta$ and let $w_s$ be the corresponding positive solution.
Set
\[
u_s=\frac{w_s}{\|w_s\|_{L^\infty(\Omega)}}.
\]
Then $u_s$ satisfies
\begin{equation}\label{equa1}
-\Delta u_s = s\frac{f(w_s)}{w_s}u_s \quad \text{in }\Omega,
\qquad
u_s=0 \quad \text{on }\partial\Omega,
\end{equation}
and $\|u_s\|_\infty=1$. Since $f(t)/t$ is decreasing and
\[
\lim_{t\to\infty}\frac{f(t)}{t}=\theta,
\]
for every $\varepsilon>0$ there exists $T>0$ such that
\[
\frac{f(t)}{t}\le \theta+\varepsilon
\qquad \text{for all } t\ge T.
\]
Fix $\varepsilon>0$ so small that
\[
s(\theta+\varepsilon)<\lambda_1.
\]

We split $\Omega$ into two regions:
\[
\Omega_1=\{x\in\Omega: w_s(x)\ge T\},
\qquad
\Omega_2=\{x\in\Omega: w_s(x)<T\}.
\]

If $x\in\Omega_1$, then
\[
s\frac{f(w_s(x))}{w_s(x)}u_s(x) 
\le 
s(\theta+\varepsilon)
\le 
\frac{\lambda_1}{\theta} (\theta+\varepsilon).
\]

If $x\in\Omega_2$, then $w_s(x)\in [0,T]$, hence 
$f(w_s(x))\le C_T$, for some constant $C_T>0$ 
independent on $s$.
Therefore,
\[
s\frac{f(w_s(x))}{w_s(x)}u_s(x)
=s \frac{f(w_s(x))}{\|w_s\|_{L^\infty(\Omega)}}\le \frac{\lambda_1}{\theta}C_T
\]
for all $s$ sufficiently close to $(\lambda_1/\theta)^-$,
because by item $(iii)$,
$\|w_s\|_{L^\infty(\Omega)}\to \infty$ as $s\to (\lambda_1/\theta)^-$.
Hence the right-hand side of \eqref{equa1} is uniformly
bounded in $L^\infty(\Omega)$
for $s$ close to $(\lambda_1/\theta)^-$. 

By elliptic regularity we deduce that
$\{u_s\}$ is bounded in $C^{1,\alpha}(\overline{\Omega})$ for some
$\alpha\in(0,1)$. Hence, up to a subsequence,
\begin{equation}\label{conv1}
	u_s \to u \quad \text{in } C^1(\overline{\Omega})
\end{equation}
as $s\to (\lambda_1/\theta)^-$. By \eqref{f1} and item (iii), passing to the limit in the weak formulation of \eqref{equa1} as $s\to (\lambda_1/\theta)^-$, gives
\[
-\Delta u = \lambda_1 u \quad \text{in }\Omega,
\qquad
u=0 \quad \text{on }\partial\Omega,
\qquad
\|u\|_\infty=1.
\]
Thus $u=\varphi_1$, the first eigenfunction associated with $\lambda_1$. 
As the limit $\varphi_1$ in \eqref{conv1} does not depend on the choice of a subsequence, we have for the whole sequence
\[
u_s\to \varphi_1
\qquad\text{in } C^1(\overline{\Omega})
\quad\text{as } s\to (\lambda_1/\theta)^-.
\]
Consequently
\[
|\nabla u_s(x)|\to |\nabla \varphi_1(x)| \quad \mbox{for all} \ x\in\Omega,
\]
that is,
\begin{equation}\label{conv2}
	\frac{|\nabla w_s(x)|}{\|w_s\|_{L^\infty(\Omega)}}\to |\nabla \varphi_1(x)| \quad \mbox{for all} \ x\in\Omega.
\end{equation}
It is a consequence of item $(iii)$ and \eqref{conv2} that
\begin{equation}\label{conv3}
	|\nabla w_s(x)|\to\infty \ \mbox{for all} \ x\in A_{\varphi_1}:=\{y\in\Omega: |\nabla \varphi_1(y)|> 0\}.
\end{equation}
Taking into account that the set of critical points of $\varphi_1$ has zero $N$-dimensional measure (see for instance \cite{Naber2017, Hardt1989}), the result follows from \eqref{conv3}.

\medskip

\noindent
{\bf Case 2: $\theta=0$.}

Assume that there exist $p\in(1,2)$ and $c_0>0$ such that
\begin{equation}\label{hipo}
\lim_{t\to\infty}\frac{f(t)}{t^{p-1}}=c_0.
\end{equation}
Define
\[
a_s := (c_0 s)^{\frac1{2-p}}
\qquad\text{and}\qquad
v_s:=\frac{w_s}{a_s}.
\]
Since \eqref{hipo} holds, there exists $T>0$ such that
\[
\frac{c_0}{2}\, t^{p-1}\le f(t)\le 2c_0\, t^{p-1}
\qquad\text{for all }t\ge T.
\]
Moreover, it is also clear that there exists $C>0$ such that
\[
0\le f(t)\le C(1+t^{p-1})
\qquad\text{for all }t\ge0.
\]

\medskip

Recall from the proof of item $(iii)$ that 
\[
w_s\ge \eta_s \varphi_1 \qquad\text{in }\Omega,
\]
Therefore
\[
\|w_s\|_{L^\infty(\Omega)}\ge \eta_s.
\]
Now we estimate $\eta_s$ in terms of $a_s$. Indeed, note that 
\[
\lim_{s\to \infty}	\frac{\psi(\eta_s)}{\eta_s^{p-2}}=c_0.
\]
The definition of $\eta_s$ gives
\[
\lim_{s\to \infty}	\frac{\lambda_1\eta_s^{2-p}}{s}=c_0,
\]
hence
\[
\lim_{s\to \infty}	\frac{\eta_s}{\lambda_1^{-\frac1{2-p}} (c_0 s)^{\frac1{2-p}}}=\lim_{s\to \infty}	\frac{\eta_s}{ \lambda_1^{-\frac1{2-p}} a_s.}=1,
\]
so there exist $c_1>0$ and $s_1>0$ such that
\[
\|w_s\|_{L^\infty(\Omega)}\ge c_1 a_s
\qquad\text{for all }s\ge s_1.
\]
Equivalently,
\[
\|v_s\|_{L^\infty(\Omega)}\ge c_1
\qquad\text{for all }s\ge s_1.
\]

\medskip
Now we will provide a suitable supersolution to bound $w_s$ from above. Let $e\in C^{2,\alpha}(\overline{\Omega})$ be the unique solution of
\[
\begin{cases}
	-\Delta e =1 & \text{in }\Omega,\\
	e=0 & \text{on }\partial\Omega.
\end{cases}
\]
Fix $M>0$ and define
\[
W_s:= M a_s e.
\]
Then
\[
-\Delta W_s = M a_s.
\]
Note that
\[
s f(W_s)\le C s\bigl(1+W_s^{p-1}\bigr)
= C s + C s (M a_s e)^{p-1}.
\]
Since
\[
s a_s^{p-1}=\frac{a_s}{c_0},
\]
we conclude that
\[
s f(W_s)\le C s + \frac{C}{c_0} M^{p-1} e^{p-1} a_s.
\]
Hence
\[
-\Delta W_s - s f(W_s)
\ge
a_s\left(
M - \frac{Cs}{a_s} - \frac{C}{c_0}M^{p-1}\|e\|_\infty^{p-1}
\right).
\]
Since $1<p<2$, we have
\begin{equation}\label{-1-1}
	\lim_{s\to\infty}	\frac{s}{a_s}=\lim_{s\to\infty} \frac{s}{(c_0 s)^{1/(2-p)}} = 0.
\end{equation}

Now choose $M>0$ large enough so that
\[
M>\frac{2C}{c_0}M^{p-1}\|e\|_\infty^{p-1}.
\]
Then for $s$ large enough,
\[
M - \frac{Cs}{a_s} - \frac{C}{c_0}M^{p-1}\|e\|_\infty^{p-1}>0,
\]
and therefore
\[
-\Delta W_s \ge s f(W_s)\qquad\text{in }\Omega.
\]
Thus $W_s$ is a supersolution of the problem solved by $w_s$. By comparison,
\[
w_s\le W_s = M a_s e \qquad\text{in }\Omega
\]
for all large $s$. Hence there exists $c_2>0$ such that
\[
\|w_s\|_{L^\infty(\Omega)}\le c_2 a_s
\qquad\text{for all large }s.
\]
Equivalently,

\begin{equation}\label{-10}
	\|v_s\|_{L^\infty(\Omega)}\le c_2
	\qquad\text{for all large }s.
\end{equation}

\medskip

Since $w_s=a_s v_s$, we have that
\[
-\Delta v_s=\frac{s}{a_s}f(a_s v_s).
\]
From
\[
\frac{s}{a_s}f(a_s v_s)
\le C\frac{s}{a_s}+ C s a_s^{p-2} v_s^{p-1}
\]
and
\[
s a_s^{p-2}=\frac1{c_0},
\]
we deduce that
\[
\frac{s}{a_s}f(a_s v_s)
\le C\frac{s}{a_s}+\frac{C}{c_0}v_s^{p-1}.
\]
By \eqref{-1-1} and \eqref{-10}, $(v_s)$ is uniformly bounded in $L^\infty(\Omega)$, and therefore the right-hand side is uniformly bounded in $L^\infty(\Omega)$. Elliptic regularity yields
\[
\|v_s\|_{C^{1,\alpha}(\overline{\Omega})}\le C
\]
for some $\alpha\in(0,1)$ and some constant $C>0$ independent of $s$. Thus, 
\[
v_s\to v \qquad\text{in } C^1(\overline{\Omega}),
\]
Moreover, since 

\[
0<c_1\le \|v_s\|_{L^\infty(\Omega)}\le c_2,
\]
we have
\[
\|v\|_{L^\infty(\Omega)}\ge c_1>0.
\]
So $v\not\equiv0$.

\medskip

We claim that
\[
\frac{s}{a_s}f(a_s v_s)\to v^{p-1}
\qquad\text{pointwisely in }\Omega.
\]

Indeed, if $v(x)>0$ then $v_s(x)\to v(x)>0$, so
\[
a_s v_s(x)\to\infty.
\]
Therefore
\[
\lim_{s\to\infty}\frac{f(a_s v_s(x))}{(a_s v_s(x))^{p-1}}	= c_0,
\]
and so
\[
\lim_{s\to\infty}	\frac{s}{a_s}f(a_s v_s(x))=\lim_{s\to\infty}c_0 s a_s^{p-2} v_s(x)^{p-1}
=\lim_{s\to\infty}v_s(x)^{p-1}=v(x)^{p-1}.
\]

If $v(x)=0$ then $v_s(x)\to0$, and by \eqref{-1-1} we get 
\[
0\le \frac{s}{a_s}f(a_s v_s(x))
\le C\frac{s}{a_s}+\frac{C}{c_0}v_s(x)^{p-1}\to 0
=v(x)^{p-1}.
\]
So the pointwise convergence is proved.

The last estimate also shows a uniform $L^\infty$ bound on the right-hand side. Hence, for every test function $\phi\in C_c^\infty(\Omega)$,
\[
\int_\Omega \nabla v_s\cdot \nabla \phi\,dx
=
\int_\Omega \frac{s}{a_s}f(a_s v_s)\phi\,dx
\longrightarrow
\int_\Omega v^{p-1}\phi\,dx.
\]
Since $v_s\to v$ in $C^1(\overline{\Omega})$, we conclude that $v$ is a nontrivial nonnegative solution of
\[
\begin{cases}
	-\Delta u = u^{p-1} & \text{in }\Omega,\\
	u=0 & \text{on }\partial\Omega.
\end{cases}
\]
The strong maximum principle implies that
\[
v>0 \text{ in }\Omega \quad \mbox{and} \quad \partial_\nu v<0
\text{ on } \partial\Omega,
\]
where $\nu$ is the outer unit normal. Fixing $y_0\in\partial\Omega$, we have
\[
|\nabla v(y_0)|>0,
\]
and by continuity of $\nabla v$ on $\overline{\Omega}$, there exist an open neighborhood $U$ of $y_0$ and a constant $c_*>0$ such that
\[
|\nabla v(x)|\ge c_*
\qquad\text{for all }x\in U\cap\Omega.
\]

Then $E:=U\cap\Omega$  has positive Lebesgue measure. Since $v_s\to v$ in $C^1(\overline{\Omega})$, we have
\[
|\nabla v_s(x)|\ge \frac{c_*}{2}
\qquad\text{for all }x\in E
\]
for all sufficiently large $s$. Therefore, recalling that $w_s=a_s v_s$, we have
\[
|\nabla w_s(x)|
=
a_s |\nabla v_s(x)|
\ge
\frac{c_*}{2}a_s
\qquad\text{for all }x\in E.
\]
Since
\[
\lim_{s\to\infty}	a_s=\lim_{s\to\infty}(c_0 s)^{\frac1{2-p}}\to\infty,
\]
it follows that
\[
|\nabla w_s(x)|\to\infty
\qquad\text{for every }x\in E.
\]
This completes the proof.

\end{proof}

\begin{rem}\label{rem1}\strut
\begin{itemize}

\item[$(i)$] It is important to point out that the assumption \eqref{g1}, \eqref{g2} and items $(iii)-(iv)$ of Lemma \ref{l1} imply $Q(s)\to 0$ as $s\to (\lambda_1/\beta)^+$ and $Q(s)\to\infty$ as $s\to (\lambda_1/\theta)^-$.
\item[$(ii)$] Another immediate consequence of Lemma \ref{l1}$(iii)$ is that the $L^p, L^\infty, C_0^1$ and $W_0^{1, p}$ norms of $w_s$ tend to zero as $s\to (\lambda_1/\beta)^+$, and to infinity as $s\to (\lambda_1/\theta)^-$, for any $p\geq 1$.
\end{itemize}
\end{rem}

Fix $0\le t'<t''<\infty$ such that $a(t')=a(t'')=0$ and $a(\alpha)>0$ for $\alpha\in (t',t'')$. In view of Lemma \ref{lem:unified-minimizer}, for each $\lambda>0$, we introduce the admissible set
\begin{equation}\label{def:Dilambda}
D_{\lambda}:=\left\{\alpha\in(t',t'')\ :\ \frac{\lambda_1}{\beta}<\frac{\lambda}{a(\alpha)}<\frac{\lambda_1}{\theta}\right\}
=\left\{\alpha\in(t',t'')\ :\ \frac{\theta}{\lambda_1}\lambda<a(\alpha)<\frac{\beta}{\lambda_1}\lambda\right\}.
\end{equation}

Note that $D_{\lambda}=(t',t'')$ if $\beta=\infty$ and $\theta=0$. Also $D_{\lambda}=\emptyset$ if $\theta>0$ and $\max_{\alpha\in [t',t'']}a(\alpha)\le \frac{\theta}{\lambda_1}\lambda$.
More generally, we have
$$D_{\lambda}=\bigcup_{j\in \mathcal{J}_{\lambda}} I_{j,\lambda},$$ where $\{I_{j,\lambda}: j \in \mathcal{J}_{\lambda}\}$ is a family of disjoint open intervals.
\vskip.3cm

\begin{rem}\label{defDlbd}
	Write $A=\max_{[t', t'']}a$. From the definition of $D_\lambda$ one can easily see that $D_\lambda\neq\emptyset$ if, and only if, $\lambda\in (0, \frac{\lambda_1}{\theta}A)$, where $\frac{\lambda_1}{\theta}A=\infty$ if $\theta=0$. So, from now on, we assume that $\lambda$ belongs to this interval. 
\end{rem}

We shall consider the following auxiliary problem, for $\alpha\in D_\lambda$:
\begin{equation}
\tag{$P_{\alpha}$}\label{palpha}
\begin{cases}
- a(\alpha)\,\Delta u = \lambda f(u) & \text{in } \Omega, \\[6pt]
u > 0 & \text{in } \Omega, \\[6pt]
u = 0 & \text{on } \partial\Omega.
\end{cases}
\end{equation}
This problem is nothing but \eqref{ws} with $s=\frac{\lambda}{a(\alpha)}$.
Lemma \ref{lem:unified-minimizer} yields the following result:
\begin{prop}\label{p1}
 For every $\alpha\in D_{\lambda}$, the problem \eqref{palpha} has a unique
solution $u_\alpha$, which corresponds to $w_{\lambda/a(\alpha)}$, the unique positive global minimizer of $\Phi_{\lambda/a(\alpha)}$.
Moreover, if $K\subset D_{\lambda}$ is a compact set then there exists $c<0$ such that
$\Phi_{\lambda/a(\alpha)}(u_\alpha)<c$ for all $\alpha\in K$.
\end{prop}

Let us deal now with
\begin{equation}\label{Pi}
P(\alpha):=g(u_\alpha), \quad \mbox{ for } \alpha\in D_{\lambda}.
\end{equation}
The following result is a direct consequence of the equality $u_\alpha = w_{\lambda/a(\alpha)}$:

\begin{lem}\label{l2}
Let $\alpha\in D_{\lambda}$. Then
\[
P(\alpha)=Q\!\left(\frac{\lambda}{a(\alpha)}\right).
\]
\end{lem}

Let $I \subset D_{\lambda}$ be an open interval. We shall say that $I$ is of type (see Figures \ref{fig01}, \ref{fig02} and \ref{fig03}):

\begin{enumerate}
    \item $I_{\infty,\infty}$ if $a(\alpha)\to \frac{\theta}{\lambda_1}\lambda$ as $\alpha\to \inf I$ and $\alpha\to \sup I$.
     \item $I_{\infty,0}$ if  $a(\alpha)\to \frac{\theta}{\lambda_1}\lambda$ as $\alpha\to \inf I$ and $a(\alpha)\to \frac{\beta}{\lambda_1}\lambda$ as $\alpha\to \sup I$.
        \item $I_{0,\infty}$ if  $a(\alpha)\to \frac{\beta}{\lambda_1}\lambda$ as $\alpha\to \inf I$ and $a(\alpha)\to \frac{\theta}{\lambda_1}\lambda$ as $\alpha\to \sup I$.
          \item $I_{0,0}$ if $a(\alpha)\to \frac{\beta}{\lambda_1}\lambda$ as $\alpha\to \inf I$ and $\alpha\to \sup I$.
\end{enumerate}

	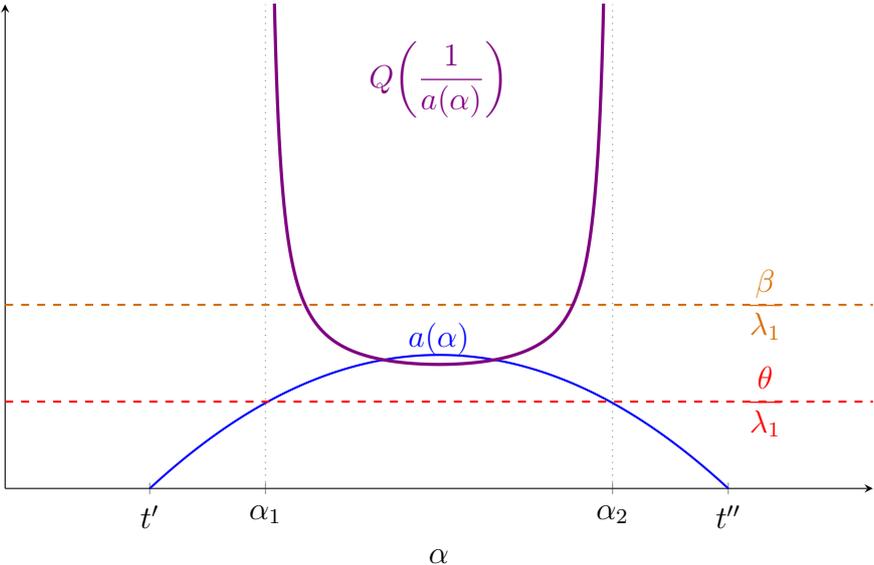
\begin{figure}[H]
		\centering
	\begin{tikzpicture}
		\begin{axis}[
			width=13cm,
			height=8cm,
			xmin=0, xmax=6,
			ymin=0, ymax=5.8,
			axis lines=left,
			xlabel={$\alpha$},
			xtick={1,1.8,4.2,5},
			xticklabels={$t'$, $\alpha_1$, $\alpha_2$, $t''$},
			ytick=\empty,
			clip=true,
			samples=400
			]
			
			\addplot[
			blue,
			thick,
			domain=1:5
			] {1.6*(x-1)*(5-x)/4};
			
			\addplot[red, dashed, thick] coordinates {(0,1.04) (6,1.04)};
			\addplot[orange!85!black, dashed, thick] coordinates {(0,2.2) (6,2.2)};
			
			\addplot[
			violet,
			very thick,
			domain=1.83:4.17
			] {1 + 0.7/((x-1.8)*(4.2-x))};
			
			\addplot[gray, dotted] coordinates {(1.8,0) (1.8,5.8)};
			\addplot[gray, dotted] coordinates {(4.2,0) (4.2,5.8)};
			
			\node[blue] at (axis cs:3,1.8) {$a(\alpha)$};
			\node[red, anchor=west] at (axis cs:5.05,1.04) {$\dfrac{\theta}{\lambda_1}$};
			\node[orange!85!black, anchor=west] at (axis cs:5.05,2.2) {$\dfrac{\beta}{\lambda_1}$};
			\node[violet] at (axis cs:3,4.9) {$Q\!\left(\dfrac{1}{a(\alpha)}\right)$};
			
		\end{axis}
	\end{tikzpicture}
	\caption{$(\alpha_1,\alpha_2)$ is of type $I_{\infty,\infty}$. Here $\lambda=1$ to simplify the figure.}
	\label{fig01}
	\end{figure}
	\begin{figure}[H]
	\centering
\begin{tikzpicture}
	\begin{axis}[
		width=14cm,
		height=8cm,
		xmin=0, xmax=6,
		ymin=0, ymax=6,
		axis lines=left,
		xlabel={$\alpha$},
		xtick={1,1.45,2.15,3.85,4.55,5},
		xticklabels={$t'$, $\alpha_1$, $\alpha_2$, $\alpha_3$, $\alpha_4$, $t''$},
		ytick=\empty,
		clip=true,
		samples=400
		]
		
		\addplot[
		blue,
		thick,
		domain=1:5
		] {2.8*(x-1)*(5-x)/4};
		
		\addplot[red, dashed, thick] coordinates {(0,1.12) (6,1.12)};
		\addplot[orange!85!black, dashed, thick] coordinates {(0,2.3) (6,2.3)};
		
		\addplot[
		violet,
		very thick,
		domain=1.48:2.12
		] {(2.15-x)/((x-1.45)+0.015)};
		
		\addplot[
		violet,
		very thick,
		domain=3.88:4.52
		] {(x-3.85)/((4.55-x)+0.015)};
		
		\addplot[gray, dotted] coordinates {(1.45,0) (1.45,6)};
		\addplot[gray, dotted] coordinates {(2.15,0) (2.15,6)};
		\addplot[gray, dotted] coordinates {(3.85,0) (3.85,6)};
		\addplot[gray, dotted] coordinates {(4.55,0) (4.55,6)};
		
		\node[blue] at (axis cs:3,2.9) {$a(\alpha)$};
		\node[red, anchor=west] at (axis cs:5.05,1.12) {$\dfrac{\theta}{\lambda_1}$};
		\node[orange!85!black, anchor=west] at (axis cs:5.05,2.3) {$\dfrac{\beta}{\lambda_1}$};
		\node[violet] at (axis cs:3,5.2) {$Q\!\left(\dfrac{1}{a(\alpha)}\right)$};
		
	\end{axis}
\end{tikzpicture}
\caption{$(\alpha_1,\alpha_2)$ is of type $I_{\infty,0}$, while $(\alpha_3,\alpha_4)$ is of type $I_{0,\infty}$.}
\label{fig02}
	\end{figure}
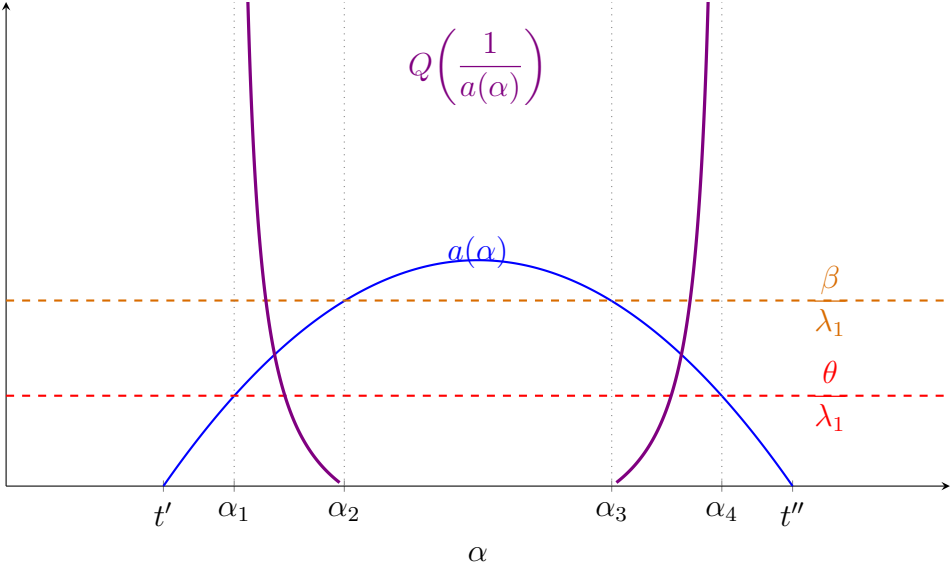
	
	\begin{figure}[H]
		\centering
		\begin{tikzpicture}
			\begin{axis}[
				width=15cm,
				height=8cm,
				xmin=0, xmax=10,
				ymin=0, ymax=6.4,
				axis lines=left,
				xlabel={$\alpha$},
				xtick={1,1.65,2.35,3.35,4.15,5.85,6.65,7.65,8.35,9},
				xticklabels={$t'$, $\alpha_1$, $\alpha_2$, $\alpha_3$, $\alpha_4$, $\alpha_5$, $\alpha_6$, $\alpha_7$, $\alpha_8$, $t''$},
				ytick=\empty,
				clip=true,
				samples=400
				]
				
				\addplot[red, dashed, thick] coordinates {(0,0.90) (10,0.90)};
				\addplot[orange!85!black, dashed, thick] coordinates {(0,1.90) (10,1.90)};
				
				\addplot[
				blue,
				very thick,
				smooth
				] coordinates {
					(1.00,0.00)
					(1.25,0.22)
					(1.45,0.55)
					(1.65,0.90)   
					(1.95,1.45)
					(2.15,1.6)
					(2.35,1.90)   
					(2.70,2.20)
					(3.00,2.30)
					(3.35,1.90)   
					(3.75,1.45)
					(4.15,1.90)   
					(5.00,2.30)
					(5.85,1.90)   
					(6.25,1.45)
					(6.65,1.90)   
					(7.00,2.30)
					(7.30,2.20)
					(7.65,1.90)   
					(7.85,1.6)
					(8.05,1.45)
					(8.35,0.90)   
					(8.55,0.55)
					(8.75,0.22)
					(9.00,0.00)
				};
				
				
				\addplot[
				violet,
				very thick,
				domain=1.68:2.32
				] {0.22*(2.35-x)/(x-1.65)};
				
				\addplot[
				violet,
				very thick,
				domain=3.36:4.14
				] {2.2*(x-3.35)*(4.15-x)};
				
				\addplot[
				violet,
				very thick,
				domain=5.86:6.64
				] {2.2*(x-5.85)*(6.65-x)};
				
				\addplot[
				violet,
				very thick,
				domain=7.68:8.32
				] {0.22*(x-7.65)/(8.35-x)};
				
				\addplot[gray, dotted] coordinates {(1.65,0) (1.65,6.4)};
				\addplot[gray, dotted] coordinates {(2.35,0) (2.35,6.4)};
				\addplot[gray, dotted] coordinates {(3.35,0) (3.35,6.4)};
				\addplot[gray, dotted] coordinates {(4.15,0) (4.15,6.4)};
				\addplot[gray, dotted] coordinates {(5.85,0) (5.85,6.4)};
				\addplot[gray, dotted] coordinates {(6.65,0) (6.65,6.4)};
				\addplot[gray, dotted] coordinates {(7.65,0) (7.65,6.4)};
				\addplot[gray, dotted] coordinates {(8.35,0) (8.35,6.4)};
				
				\node[blue] at (axis cs:5.0,3.0) {$a(\alpha)$};
				\node[red, anchor=west] at (axis cs:9.05,0.90) {$\dfrac{\theta}{\lambda_1}$};
				\node[orange!85!black, anchor=west] at (axis cs:9.05,1.90) {$\dfrac{\beta}{\lambda_1}$};
				\node[violet] at (axis cs:5.0,5.45) {$Q\!\left(\dfrac{1}{a(\alpha)}\right)$};
				
			\end{axis}
		\end{tikzpicture}
	\caption{$(\alpha_1,\alpha_2)$ is of type $I_{\infty,0}$, $(\alpha_3,\alpha_4)$ and $(\alpha_5,\alpha_6)$ are of type $I_{0,0}$, while $(\alpha_7,\alpha_8)$ is of type $I_{0,\infty}$.}
	\label{fig03}
	\end{figure}
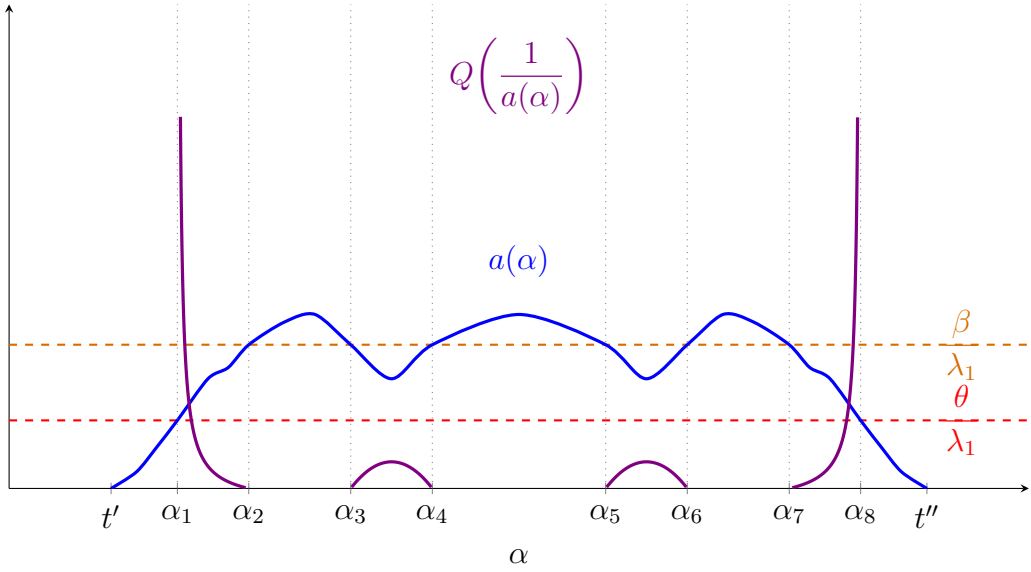

\begin{lem}\label{l3}
 Let $I\subset D_{\lambda}$ be a maximal open interval. Then $I$ is of type $I_{\infty,\infty}$, $I_{\infty,0}$, $I_{0,\infty}$ or $I_{0,0}$. Moreover  $P:D_{\lambda}\to(0,\infty)$ is continuous and
\[
\lim_{\alpha\to \inf I}P(\alpha)=\lim_{\alpha\to \sup I}P(\alpha)=+\infty,\ \ \mbox{if}\ \ I \mbox{ is of type }I_{\infty,\infty},
\]
\[
\lim_{\alpha\to \inf I}P(\alpha)=+\infty,
\qquad
\lim_{\alpha\to \sup I}P(\alpha)=0, \ \ \mbox{if}\ \ I \mbox{ is of type } I_{\infty,0},
\]
\[
\lim_{\alpha\to \inf I}P(\alpha)=0,
\qquad
\lim_{\alpha\to \sup I}P(\alpha)=+\infty,  \ \ \mbox{if}\ \ I\mbox{ is of type } I_{0,\infty},
\]
\[
\lim_{\alpha\to \inf I}P(\alpha)=\lim_{\alpha\to \sup I}P(\alpha)=0,\ \ \mbox{if}\ \ I\mbox{ is of type } I_{0,0}.
\]

Furthermore,

\begin{enumerate}
    \item[$(i)$] If $I$ is of type $I_{\infty,0}$  then there exists $J \subset D_{\lambda}$ of type $I_{0,\infty}$ such that $\sup I\leq \inf J$ (see Figure \ref{fig02}).
    \item[$(ii)$] If $I$ is of type $I_{0,0}$, then there exists $J_1\subset D_{\lambda}$ of type $I_{\infty,0}$. In particular,  there exists $J_2\subset D_\lambda$ of type $I_{0,\infty}$ such that $\sup J_1\leq \inf J_2$  (see Figure \ref{fig03}).
\end{enumerate}
\end{lem}

\begin{proof}
Since $a$ is continuous and $a(\alpha)>0$ on $(t',t'')$, the map
$\alpha\mapsto \lambda/a(\alpha)$ is continuous on $D_{\lambda}$. By the definition of $D_{\lambda}$ and since $I$ is maximal, it is also clear that $I$ is of one of these four types. The limits of $P$ at $\inf I$ and $\sup I$ come from Lemmas \ref{l1} and \ref{l2}, and Remark \ref{rem1}.
\vskip.3cm
Now we prove $(i)$: if $I$ is of type $I_{\infty,0}$, then it is clear that  $\alpha(\sup I)= \frac{\beta}{\lambda_1}\lambda$ and $\sup I<t''$ (since $a$ is close to zero near $t''$). Let $\alpha_0=\sup\{ \alpha\in[\sup I,t'']:\  a(\alpha)= \frac{\beta}{\lambda_1}\lambda\}$. Since $a$ is close to zero near $t''$ it follows that $\alpha_0<t''$. Moreover, if $\alpha\in (\alpha_0,t'']$, then $a(\alpha)< \frac{\beta}{\lambda_1}\lambda$. Now let $\alpha_1=\inf\{ \alpha\in[\alpha_0,t'']:\  a(\alpha)= \frac{\theta}{\lambda_1}\lambda\}$. Then it is clear that $\alpha_0<\alpha_1\leq t''$ and the interval $J:=(\alpha_0,\alpha_1)$ is of type $I_{0,\infty}$. Also by definition $\sup I\leq \inf J$.
\vskip.3cm
$(ii)$: If $I$ if of type $I_{0,0}$, then $a(\inf I)=a(\sup I)=\frac{\beta}{\lambda_1}\lambda$. Let $\alpha_1=\inf\{ \alpha\in[t',\inf I]:\  a(\alpha)= \frac{\beta}{\lambda_1}\lambda\}$. Since $a$ is close to zero near $t'$ it follows that $\alpha_1>t'$. Now define
$\alpha_0=\sup\{ \alpha\in[t',\alpha_1]:\  a(\alpha)= \frac{\theta}{\lambda_1}\lambda\}$, then it is clear that $\alpha_0<\alpha_1$ and the interval $J_1=(\alpha_0,\alpha_1)$ is of type $I_{\infty,0}$. To conclude we apply item $(i)$.

\end{proof}

\section{Existence of Fixed Points of $P$}
\subsection{The General Case} \label{secfg}
In this section we study the existence of fixed points of $P$. Recall that
\[
P(\alpha)=Q\!\left(\frac{\lambda}{a(\alpha)}\right), \alpha\in D_\lambda.
\]

First note that, if $a(\alpha)= (\beta/\lambda_1)\lambda$, we can define $P(\alpha)=0$, so that, by Lemma \ref{l3}, $P$ is continuous  in the extended set
$$\widetilde {D}_\lambda=D_\lambda\cup\{\alpha\in \overline{D_\lambda}: \ a(\alpha)=(\beta/\lambda_1)\lambda\}.$$
By Remark \ref{defDlbd} we have that $\widetilde{D}_\lambda\neq \emptyset$ if, and only if, $\lambda\in (0, \frac{\lambda_1}{\theta}A)$. Note that if $\beta=\infty$ then $\widetilde{D}_\lambda=D_\lambda$.  Moreover $P$ is coercive in $\widetilde {D}_\lambda$ in the sense that, if there exists a sequence $\alpha_n\in D_\lambda$ such that $\alpha_n\to \alpha$ and $a(\alpha)=(\theta/\lambda_1)\lambda$, then Lemma \ref{l3} yields $P(\alpha_n)\to \infty$. So we have proved the following:
\begin{lem}\label{coerext} $P:\widetilde {D}_\lambda\to [0,\infty)$ is continuous and coercive.
\end{lem}

Write $T(\lambda,\alpha)=P(\alpha)-\alpha=Q\!\left(\frac{\lambda}{a(\alpha)}\right)-\alpha$, for $\lambda>0$ and $\alpha\in \tilde {D}_\lambda$. By Lemma \ref{coerext} we have that
\begin{equation}\label{defc}
c(\lambda):=\inf\{T(\lambda,\alpha), \alpha \in \widetilde {D}_\lambda\}
\end{equation}
is attained at some $\alpha_\lambda\in \tilde {D}_\lambda$.

\begin{lem} \label{clbd}The map $c$, given by \eqref{defc},  is continuous in $(0,A\lambda_1/\theta)$. Moreover
\begin{equation*}
\lim_{\lambda\to 0^+}c(\lambda)<0, \quad \mbox{and} \quad  \lim_{\lambda\to \left(A\lambda_1/\theta\right)^-}c(\lambda)=\infty.
\end{equation*}
\end{lem}

\begin{proof}
Let $\mu_n\to\lambda_0 \in (0,A\lambda_1/\theta)$ and $\alpha_0\in \widetilde D_{\lambda_0}$ be such that
\[
c(\lambda_0)=T(\lambda_0,\alpha_0),
\]

By continuity  it is clear that there exists a sequence $\alpha_n\to\alpha_0$ with $\alpha_n\in \widetilde D_{\mu_n}$. Hence, since $T$ is continuous we have
\[
T(\mu_n,\alpha_n)\to T(\lambda_0,\alpha_0)=c(\lambda_0).
\]
From $c(\mu_n)\le T(\mu_n,\alpha_n)$, we obtain that
$\limsup_{n\to\infty} c(\mu_n)\le c(\lambda_0)$.

Take now
$\beta_n\in \widetilde D_{\mu_n}$  such that
$c(\mu_n)=T(\mu_n,\beta_n)$.
Since $[t',t'']$ is compact, there exists a subsequence (still denoted by $\beta_n$) and some $\beta\in X$ such that
$\beta_n\to \beta \quad \text{in } [t',t'']$.
Clearly, we have $\beta\in \widetilde D_{\lambda_0}$. Therefore, using again the continuity of $T$, we find that
\[
\liminf_{n\to\infty} c(\mu_n)
=\liminf_{n\to\infty} T(\mu_n,\beta_n)
= T(\lambda_0,\beta)\ge  c(\lambda_0).
\]
This yields $c(\mu_n) \to c(\lambda_0)$, as desired.
Thus $c$ is continuous.

Now let $\mu_n\to 0^+$. Since $a$ is continuous and $a(t')=a(t'')=0$ we can find $\alpha_n\in \widetilde D_{\mu_n}$. Moreover we can assume that $\lambda_1/\beta<\mu_n/a(\alpha_n)<\lambda_1/\beta+1/n$ and $t''-1/n<\alpha_n<t''$, for $n\in\mathbb{N}$, i.e. $\alpha_n \to t''$ and $\mu_n/a(\alpha_n) \to \lambda_1/\beta$. Therefore, by Lemma \ref{l3} we conclude that
\begin{equation*}
\limsup_{n\to \infty}c(\mu_n)\le   \lim_{n\to \infty}T(\mu_n,\alpha_n)=    \lim_{n\to\infty}\left[Q\!\left(\frac{\mu_n}{a(\alpha_n)}\right)-\alpha_n\right]=-t''<0.
\end{equation*}

Now suppose that $\mu_n\to  \left(A\lambda_1/\theta\right)^-$. If $\theta>0$, let $\alpha_0$ be such that $a(\alpha_0)=A$. Then, it is clear that $\alpha_0\in \widetilde D_{\mu_n}$ and $a(\alpha_0)/\mu_n=A/\mu_n\to \theta/\lambda_1$. Moreover, if $\alpha_n\in \widetilde{D}_{\mu_n}$ satisfies $c(\mu_n)=T(\mu_n,\alpha_n)$, then
\begin{equation*}
\frac{\mu_n}{A}\le\frac{\mu_n}{a(\alpha_n)}<\frac{\lambda_1}{\theta,}
\end{equation*}
which implies, by Lemma \ref{l2}, that
\begin{equation*}
\liminf_{n\to \infty}c(\mu_n)=  \lim_{n\to \infty}T(\mu_n,\alpha_n)=    \lim_{n\to\infty}\left[Q\!\left(\frac{\mu_n}{a(\alpha_n)}\right)-\alpha_n\right]=\infty.
\end{equation*}
\end{proof}

Next we set
\begin{equation}\label{lbdtilde}
\lambda_0=\min\{\lambda \in (0,A\lambda_1/\theta): c(\lambda)=0\}\ \ \ \mbox{and}\ \ \  \widetilde\lambda_0=\max\{\lambda \in (0,A\lambda_1/\theta): c(\lambda)=0\}.
\end{equation}

It follows from Lemma \ref{clbd} that $0<\lambda_0<\widetilde\lambda_0<(A\lambda_1)/\theta$, $c(\lambda)<0$ for $\lambda<\lambda_0$ and $c(\lambda)>0$ for $\lambda>\widetilde \lambda_0$.
\begin{prop} \label{prop0}\strut
\begin{enumerate}
\item[$(i)$] If $0<\lambda<\lambda_{0}$ then  $P$ has at least two fixed points $\alpha_{1},\alpha_{2}$ such that $t'<\alpha_{1}<\alpha_{2}<t''$.

\item[$(ii)$] If $\lambda>\widetilde\lambda_{0}$ then $P$ has no fixed points in $(t',t'')$.
\end{enumerate}

\end{prop}
\begin{proof} $(i)$:  It is clear, by Lemma \ref{l3}, that if $D_\lambda$ contains a maximal open interval of type $I_{\infty,0}$ or $I_{0,\infty}$ or $I_{0,0}$, then $P$ has at least two fixed points in $D_\lambda$ (see Figure \ref{fig04}). Thus, let us assume that the only maximal open intervals of $D_\lambda$ are of type $I_{\infty,\infty}$. Since $c(\lambda)<0$, there exists $\alpha_0\in \widetilde D_\lambda=D_\lambda$ such that $T(\lambda,\alpha_0)<0$. Thus $P(\alpha_0)<\alpha_0$. On the other hand we have, by Lemma \ref{l3},  that if $I \subset D_\lambda$ is a maximal interval of type $I_{\infty,\infty}$ then $P(\alpha)\to \infty$ as $\alpha\to \inf I $ and $\alpha\to \sup I $. Then clearly $P$ has at least two fixed points in $D_\lambda$ (see Figure \ref{fig05}).

$(ii)$: If $\lambda>\widetilde\lambda_{0}$, then $0<c(\lambda)\le T(\lambda,
\alpha)$ for all $\alpha\in \widetilde D_\lambda$, which implies that $P(\alpha)>\alpha$ for all $\alpha\in \widetilde D_\lambda$, so $P$ has no fixed points. Of course, if $D_\lambda=\emptyset$ the proof is trivial.
\end{proof}

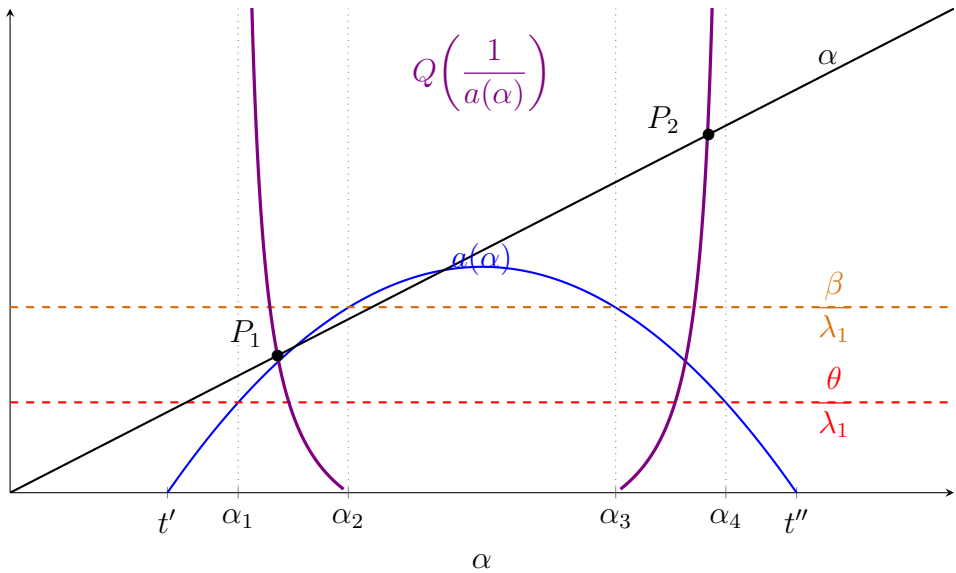
\begin{figure}[H]
	\centering
	\begin{tikzpicture}
		\begin{axis}[
			width=14cm,
			height=8cm,
			xmin=0, xmax=6,
			ymin=0, ymax=6,
			axis lines=left,
			xlabel={$\alpha$},
			xtick={1,1.45,2.15,3.85,4.55,5},
			xticklabels={$t'$, $\alpha_1$, $\alpha_2$, $\alpha_3$, $\alpha_4$, $t''$},
			ytick=\empty,
			clip=true,
			samples=400
			]
			
			\addplot[
			blue,
			thick,
			domain=1:5
			] {2.8*(x-1)*(5-x)/4};
			
			\addplot[red, dashed, thick] coordinates {(0,1.12) (6,1.12)};
			\addplot[orange!85!black, dashed, thick] coordinates {(0,2.3) (6,2.3)};
			
			\addplot[
			violet,
			very thick,
			domain=1.48:2.12
			] {(2.15-x)/((x-1.45)+0.015)};
			
			\addplot[
			violet,
			very thick,
			domain=3.88:4.52
			] {(x-3.85)/((4.55-x)+0.015)};
			
			\addplot[
			black,
			thick,
			domain=0:6
			] {x};
			
			\addplot[gray, dotted] coordinates {(1.45,0) (1.45,6)};
			\addplot[gray, dotted] coordinates {(2.15,0) (2.15,6)};
			\addplot[gray, dotted] coordinates {(3.85,0) (3.85,6)};
			\addplot[gray, dotted] coordinates {(4.55,0) (4.55,6)};
			
			\addplot[only marks, mark=*, mark size=2pt, black] coordinates {
				(1.7,1.7)
				(4.44,4.44)
			};
			
			\node[black, above left] at (axis cs:1.673,1.673) {$P_1$};
			\node[black, above left] at (axis cs:4.327,4.327) {$P_2$};
			
			\node[blue] at (axis cs:3,2.9) {$a(\alpha)$};
			\node[red, anchor=west] at (axis cs:5.05,1.12) {$\dfrac{\theta}{\lambda_1}$};
			\node[orange!85!black, anchor=west] at (axis cs:5.05,2.3) {$\dfrac{\beta}{\lambda_1}$};
			\node[violet] at (axis cs:3,5.2) {$Q\!\left(\dfrac{1}{a(\alpha)}\right)$};
			\node[black] at (axis cs:5.2,5.4) {$\alpha$};
			
		\end{axis}
	\end{tikzpicture}
	\caption{$(\alpha_1,\alpha_2)$ has type $I_{\infty,0}$, while $(\alpha_3,\alpha_4)$ has type $I_{0,\infty}$. Fixed points at $P_1$ and $P_2$.}
	\label{fig04}
\end{figure}

\begin{figure}[H]
	\centering
	\begin{tikzpicture}
		\begin{axis}[
			width=13cm,
			height=8cm,
			xmin=0, xmax=6,
			ymin=0, ymax=5.8,
			axis lines=left,
			xlabel={$\alpha$},
			xtick={1,1.8,4.2,5},
			xticklabels={$t'$, $\alpha_1$, $\alpha_2$, $t''$},
			ytick=\empty,
			clip=true,
			samples=400
			]
			
			\addplot[
			blue,
			thick,
			domain=1:5
			] {1.6*(x-1)*(5-x)/4};
			
			\addplot[red, dashed, thick] coordinates {(0,1.03) (6,1.03)};
			\addplot[orange!85!black, dashed, thick] coordinates {(0,2.2) (6,2.2)};
			
			\addplot[
			violet,
			very thick,
			domain=1.83:4.17
			] {1 + 0.7/((x-1.8)*(4.2-x))};
			
			\addplot[
			black,
			thick,
			domain=0:5.8
			] {x};
			
			\addplot[gray, dotted] coordinates {(1.8,0) (1.8,5.8)};
			\addplot[gray, dotted] coordinates {(4.2,0) (4.2,5.8)};
			
			\addplot[
			only marks,
			mark=*,
			mark size=2pt,
			black
			] coordinates {
				(2.102,2.102)
				(4.103,4.103)
			};
			
			\node[black, above left] at (axis cs:2.102,2.102) {$P_1$};
			\node[black, above left] at (axis cs:4.103,4.103) {$P_2$};
			
			\node[blue] at (axis cs:3,1.8) {$a(\alpha)$};
			\node[red, anchor=west] at (axis cs:5.05,1.03) {$\dfrac{\theta}{\lambda_1}$};
			\node[orange!85!black, anchor=west] at (axis cs:5.05,2.2) {$\dfrac{\beta}{\lambda_1}$};
			\node[violet] at (axis cs:3,4.9) {$Q\!\left(\dfrac{1}{a(\alpha)}\right)$};
			\node[black] at (axis cs:5.05,5.2) {$\alpha$};
			
		\end{axis}
	\end{tikzpicture}
	\caption{$(\alpha_1,\alpha_2)$ has type $I_{\infty,\infty}$. Fixed points at $P_1$ and $P_2$.}
	\label{fig05}
\end{figure}
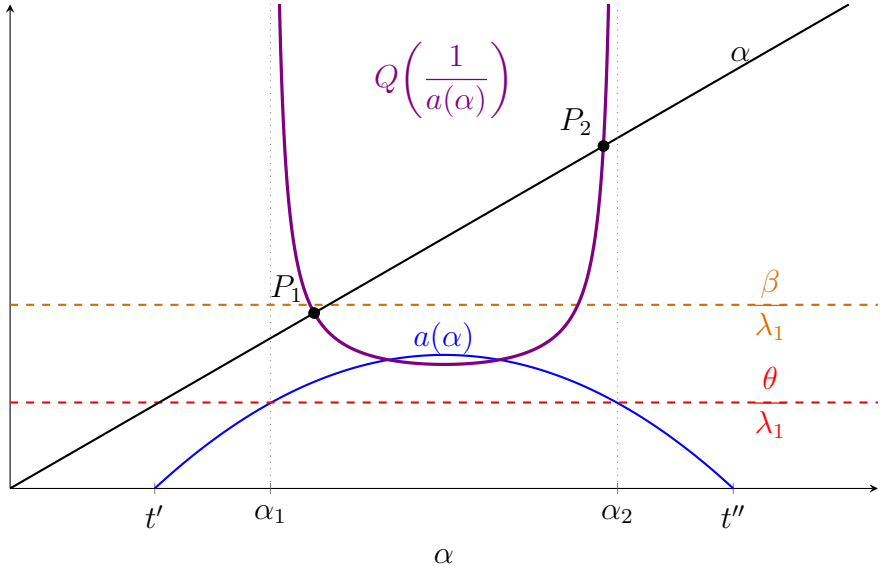
\subsection{The Case ``$Q$ is Increasing"} \label{sqincreasing}

In this section we assume that $Q$ is increasing, so that by Lemma \ref{l1} it is an increasing homeomorphism, and the following parameter is well defined:

\begin{equation}\label{def:lambda0sharp}
\lambda_{0}:=\max_{\alpha\in[t',t'']}\Big(a(\alpha)\,Q^{-1}(\alpha)\Big).
\end{equation}

For the next result recall that $A=\max\{a(\alpha):\alpha\in [t',t'']\}$.

\begin{lem}\label{lmax} $D_\lambda$ has only maximal open intervals of type $I_{\infty,\infty}$ if, and only if, $A<\frac{\lambda\beta}{\lambda_1}$.
\end{lem}
\begin{proof} Indeed, note that $A<\frac{\lambda\beta}{\lambda_1}$ if, and only if, $a(\alpha)<\frac{\lambda\beta}{\lambda_1}$ for all $\alpha\in \overline{D}_\lambda$ if, and only if,  $a(\alpha)<\frac{\lambda\beta}{\lambda_1}$ for all $\alpha\in \partial{D}_\lambda$ if, and only if, $a(\alpha)=\frac{\lambda\theta}{\lambda_1}$ for all $\alpha\in \partial{D}_\lambda$ if, and only if, $D_\lambda$ has only maximal open intervals of type $I_{\infty,\infty}$.
\end{proof}

\begin{prop}\label{prop1}\strut
\begin{enumerate}
    \item[$(i)$] If $0<\lambda<\lambda_{0}$ then  $P$ has at least two fixed points $\alpha_{1},\alpha_{2}$ such that $t'<\alpha_{1}<\alpha_{2}<t''$.
\item[$(ii)$] If $\lambda=\lambda_{0}$ then $P$ has at least one fixed point in $(t',t'')$.
\item[$(iii)$] If $\lambda>\lambda_{0}$ then $P$ has no fixed points in $(t',t'')$.
\end{enumerate}

\end{prop}

\begin{proof} As we already noticed in the proof of Proposition \ref{prop0}, if $D_\lambda$ contains a maximal open interval of type $I_{\infty,0}$ or $I_{0,\infty}$ or $I_{0,0}$, then $P$ has at least two fixed points in $D_\lambda$ (see Figure \ref{fig04}). So we focus on the case where the maximal intervals of $D_\lambda$ are of type $I_{\infty,\infty}$. By Lemma \ref{lmax} we know that this happens if, and only if, $A<\frac{\lambda\beta}{\lambda_1}$ or, equivalently, if, and only if, $\frac{A\lambda_1}{\beta}<\lambda$. 
	
	We claim that $c(\frac{A\lambda_1}{\beta})<0$, where $c(\lambda)$ was defined in \eqref{defc}. Indeed, it is clear that when $\lambda=\frac{A\lambda_1}{\beta}$, then $D_\lambda$ has a maximal open interval of type $I_{\infty,0}$,  say $J$. So, by continuity, if we take $\alpha$ close enough to $\sup J$, then $P(\alpha)-\alpha$ is negative. 
	
	Now we claim that $c(\lambda)$ is increasing in the interval $(\frac{A\lambda_1}{\beta},\frac{A\lambda_1}{\theta})$. Indeed, take $\frac{A\lambda_1}{\beta}<\mu_1<\mu_2<\frac{A\lambda_1}{\theta}$ and $\alpha_2\in D_{\mu_2}$ such that $c(\mu_2)=T(\mu_2,\alpha_2)$. Since $\frac{\mu_2\theta}{\lambda_1}<a(\alpha_2)\le A< \frac{\mu_2\beta}{\lambda_1}$, it follows that $\frac{\mu_1\theta}{\lambda_1}<\frac{\mu_2\theta}{\lambda_1}<a(\alpha_2)\le A< \frac{\mu_1\beta}{\lambda_1}$, so $\alpha_2\in D_{\mu_1}$. Thus
	\begin{eqnarray*}
		c(\mu_2)=Q\!\left(\frac{\mu_2}{a(\alpha_2)}\right)-\alpha_2>Q\!\left(\frac{\mu_1}{a(\alpha_2)}\right)-\alpha_2\ge c(\mu_1).
	\end{eqnarray*}
	
	It follows that 
	$\lambda_0=\widetilde\lambda_0=\max_{\alpha\in[t',t'']}\Big(a(\alpha)\,Q^{-1}(\alpha)\Big)$, where $\widetilde\lambda_0$ was defined in \eqref{lbdtilde}. This proves $(i)$ and $(iii)$. To prove $(ii)$ note that, since $Q^{-1}$ is strictly increasing on $(0,\infty)$ we have, for every $\alpha\in D_\lambda$,
	\begin{align}
		P(\alpha)<\alpha \mbox{ (respect. $=$, $>$)}
		&\iff
		Q\!\left(\frac{\lambda}{a(\alpha)}\right)<Q\!\left(Q^{-1}(\alpha)\right)\notag\\
		&\iff
		\frac{\lambda}{a(\alpha)}<Q^{-1}(\alpha)\notag\\
		&\iff
		\lambda<a(\alpha)Q^{-1}(\alpha) \mbox{ (respect. $=$, $>$)}. \label{eq:key-ineq}
	\end{align}
	Now fix $\alpha_0\in (t',t'')$ such that $\lambda_0=a(\alpha_0)\,Q^{-1}(\alpha_0)$. By definition, $c(\lambda_0)=0$, so it follows that $\alpha_0\in D_\lambda$, since $T(\lambda_0,\alpha)<0$ or $\infty$ in $\partial D_\lambda$, so by \eqref{eq:key-ineq} the proof is complete (see Figure \ref{fig06}).
	
	\end{proof}

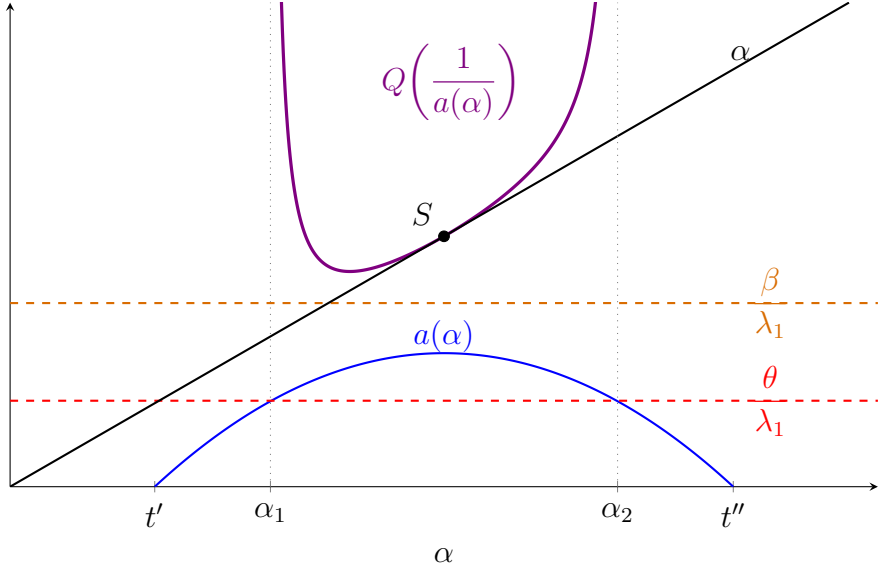
\begin{figure}[H]
	\centering
	\begin{tikzpicture}
		\begin{axis}[
			width=13cm,
			height=8cm,
			xmin=0, xmax=6,
			ymin=0, ymax=5.8,
			axis lines=left,
			xlabel={$\alpha$},
			xtick={1,1.8,4.2,5},
			xticklabels={$t'$, $\alpha_1$, $\alpha_2$, $t''$},
			ytick=\empty,
			clip=true,
			samples=400
			]
			
			\addplot[
			blue,
			thick,
			domain=1:5
			] {1.6*(x-1)*(5-x)/4};
			
			\addplot[red, dashed, thick] coordinates {(0,1.03) (6,1.03)};
			\addplot[orange!85!black, dashed, thick] coordinates {(0,2.2) (6,2.2)};
			
			\addplot[
			violet,
			very thick,
			domain=1.83:4.17
			] {x + 0.55*(x-3)^2/((x-1.8)*(4.2-x))};
			
			\addplot[
			black,
			thick,
			domain=0:5.8
			] {x};
			
			\addplot[gray, dotted] coordinates {(1.8,0) (1.8,5.8)};
			\addplot[gray, dotted] coordinates {(4.2,0) (4.2,5.8)};
			
			\addplot[
			only marks,
			mark=*,
			mark size=2pt,
			black
			] coordinates {
				(3,3)
			};
			
			\node[black, above left] at (axis cs:3,3) {$S$};
			
			\node[blue] at (axis cs:3,1.8) {$a(\alpha)$};
			\node[red, anchor=west] at (axis cs:5.05,1.03) {$\dfrac{\theta}{\lambda_1}$};
			\node[orange!85!black, anchor=west] at (axis cs:5.05,2.2) {$\dfrac{\beta}{\lambda_1}$};
			\node[violet] at (axis cs:3.05,4.85) {$Q\!\left(\dfrac{1}{a(\alpha)}\right)$};
			\node[black] at (axis cs:5.05,5.2) {$\alpha$};
			
		\end{axis}
	\end{tikzpicture}
	\caption{$(\alpha_1,\alpha_2)$ is of type $I_{\infty,\infty}$. The line $y=\alpha$ is tangent to $Q\!\left(\frac{1}{a(\alpha)}\right)$ at $S$.}
	\label{fig06}
\end{figure}


\begin{prop}\label{oscifixed} Assume the conditions of Theorem \ref{thm6}. Then $P$ has at least $2j$ fixed points.
\end{prop}

\begin{proof} Indeed, from the equivalences in \eqref{eq:key-ineq} it is sufficient to show that for $\lambda=a(\alpha)Q^{-1}(\alpha)$ we have at least $2j$ solutions, which is clear from our assumptions.
   
\end{proof}

\subsection{Some Bounds for $\lambda_0$}

In this section we prove some bounds to $\lambda_0$ in the case $g(u)=\int_\Omega |u|^\gamma dx$.

\begin{lem}\label{lbounds}
Let $v$ satisfy 
$-\Delta v=1$ in $\Omega$, $v=0$ on $\partial\Omega$,
and 
$C_e:=\Big(\int_\Omega e_1^\gamma\,dx\Big)^{1/\gamma}$.
Then:
\begin{enumerate}
\item[(a)]
For every $s>0$ such that $\lambda_1/s\in(\theta,\beta)$, we have
\[
Q(s)\ge \big(\psi^{-1}(\lambda_1/s)\big)^\gamma\,\int_\Omega e_1^\gamma\,dx
= \big(\psi^{-1}(\lambda_1/s)\big)^\gamma\,C_e^\gamma.
\]
Consequently, for every $\alpha>0$,
\begin{equation}\label{eq:Qinv-upper-general}
\boxed{\quad
Q^{-1}(\alpha)\ \le\ \frac{\lambda_1}{\psi\!\left(\alpha^{1/\gamma}/C_e\right)}.\quad}
\end{equation}
\item[(b)]
Assume moreover that $L:=\sup_{t>0}f(t)<\infty$. Then, for every $s>0$,
\[
Q(s)\le (sL)^\gamma\int_\Omega v^\gamma\,dx=(sL)^\gamma C_v,
\]
so that, for every $\alpha>0$,
\begin{equation}\label{eq:Qinv-lower-general}
\boxed{\quad
Q^{-1}(\alpha)\ \ge\ \frac{\alpha^{1/\gamma}}{L\,C_v^{1/\gamma}}\quad.}
\end{equation}
\end{enumerate}
\end{lem}

\begin{proof}
(a) Fix $s>0$ with $\lambda_1/s\in(0,\beta)$ and set
\[
\eta_s:=\psi^{-1}\!\Big(\frac{\lambda_1}{s}\Big),\qquad z_s:=\eta_s e_1.
\]
Since $0<e_1\le 1$ and $\psi$ is decreasing, $\psi(z_s)=\psi(\eta_s e_1)\ge \psi(\eta_s)=\lambda_1/s$. Hence
\[
-\Delta z_s=\lambda_1\eta_s e_1 \le s\,\psi(z_s)\,z_s=s f(z_s),
\]
so $z_s$ is a subsolution of \eqref{ws}. By comparison, $w_s\ge z_s$ in $\Omega$, and thus
\[
Q(s)=\int_\Omega w_s^\gamma\,dx\ge \int_\Omega z_s^\gamma\,dx
=\eta_s^\gamma\int_\Omega e_1^\gamma\,dx.
\]
Now let $\alpha>0$ and choose $s=Q^{-1}(\alpha)$. Then $\alpha=Q(s)\ge \eta_s^\gamma C_e^\gamma$, so
$\eta_s\le \alpha^{1/\gamma}/C_e$. Since $\psi$ is decreasing,
\[
\psi(\eta_s)\ge \psi\!\left(\alpha^{1/\gamma}/C_e\right).
\]
But $\psi(\eta_s)=\lambda_1/s$, hence
\[
\frac{\lambda_1}{s}\ge \psi\!\left(\alpha^{1/\gamma}/C_e\right)
\quad\Longrightarrow\quad
s\le \frac{\lambda_1}{\psi\!\left(\alpha^{1/\gamma}/C_e\right)}.
\]
This proves \eqref{eq:Qinv-upper-general}.

\smallskip
(b) Suppose $L=\sup_{t>0}f(t)<\infty$. Since $f(w_s)\le L$, from \eqref{ws} we have
\[
-\Delta w_s=s f(w_s)\le sL\quad\text{in }\Omega.
\]
Let $U:=sL\,v$. Then $-\Delta U=sL\ge -\Delta w_s$ in $\Omega$ and $U=w_s=0$ on $\partial\Omega$. By comparison,
$0<w_s\le U$ in $\Omega$, and hence
\[
Q(s)=\int_\Omega w_s^\gamma\,dx\le \int_\Omega (sL v)^\gamma\,dx=(sL)^\gamma\int_\Omega v^\gamma\,dx.
\]
Setting again $s=Q^{-1}(\alpha)$ yields $\alpha\le (sL)^\gamma C_v$, i.e.
$s\ge \alpha^{1/\gamma}/(LC_v^{1/\gamma})$, which is \eqref{eq:Qinv-lower-general}.
\end{proof}

\begin{cor}\label{c1}
Fix $i\in\{1,\dots,k-1\}$, and define
\[
\lambda_{0}:=\max_{\alpha\in[t',t'']}\Big(a(\alpha)\,Q^{-1}(\alpha)\Big).
\]
Then, with $C_e$ as in Lemma~\ref{lbounds},
\begin{equation}\label{eq:lambda0i-upper-general}
\boxed{\quad
\lambda_{0}
\ \le\
\lambda_1\,
\max_{\alpha\in[t',t'']}
\frac{a(\alpha)}{\psi\!\left(\alpha^{1/\gamma}/C_e\right)}
\quad}
\end{equation}
 If in addition $L=\sup_{t>0}f(t)<\infty$, then,
\begin{equation}\label{eq:lambda0i-lower-general}
\boxed{\quad
\lambda_{0}
\ \ge\
\frac{1}{L\,C_v^{1/\gamma}}\,
\max_{\alpha\in[t',t'']}\Big(a(\alpha)\,\alpha^{1/\gamma}\Big)
\quad.}
\end{equation}
\end{cor}

\begin{proof}
The estimate \eqref{eq:lambda0i-upper-general} follows by multiplying \eqref{eq:Qinv-upper-general} by $a(\alpha)$ and taking
the maximum over $\alpha\in[t',t'']$. Similarly, \eqref{eq:lambda0i-lower-general} follows from
\eqref{eq:Qinv-lower-general}.
\end{proof}

\section{Proofs of the Main Results}

Before we prove our main results, let us show the relation between solutions of \eqref{Pg} and solutions of \eqref{palpha}.

\begin{lem}\label{equfixedpoints} Assume \eqref{f1} and fix $0\le t'<t''<\infty$ such that $a(t')=a(t'')=0$ and $a(\alpha)>0$ for $\alpha\in (t',t'')$. Then, $u$ is a solution of \eqref{Pg} satisfying
\begin{equation}
\label{ineqk}
    t'<g(u)<t'',
\end{equation}
if, and only if, $u$ is the unique solution of \eqref{palpha} for some $\alpha\in (t',t'')$ with $P(\alpha)=\alpha$, i.e., $u=u_\alpha$ and $P(\alpha)=\alpha$.
\end{lem}

\begin{proof} Indeed, suppose $u$ is a solution of \eqref{Pg} satisfying \eqref{ineqk}. Write $\alpha = g(u)$ and note that $\alpha\in (t',t'')$ so $u$ is a solution \eqref{palpha}. By Proposition \ref{p1} we know that this solution is unique, so $u=u_\alpha$ and clearly $P(\alpha)=g(u_\alpha)=g(u)=\alpha$.
The converse is trivial.
\end{proof}

Now we prove our main theorems:

\begin{proof}[Proof of Theorem \ref{thm-1}]: Just take $t'=t_i$, $t''=t_{i+1}$ and apply Propositions \ref{prop0} and \ref{prop1}, and Lemma \ref{equfixedpoints}.
\end{proof}

\begin{proof}[Proof of Theorem \ref{thm6}]: It follows from Proposition \ref{oscifixed} and Lemma \ref{equfixedpoints}.
\end{proof}

To continue we need the next result.

\begin{lem}\label{hlemma}
Let $\gamma_1, \gamma_2\geq 1$, $g_1(u)=\int_\Omega |u|^{\gamma_1} dx$, and $g_2(u)=\int_\Omega |\nabla u|^{\gamma_2}dx$ for  $u \in C_0^1(\overline{\Omega})$. Then $g_1$ satisfies \eqref{g1} and $g_2$ satisfies \eqref{g2}. Moreover $Q_1(s):=g_1(w_s)$ is increasing for all $\gamma_1\geq 1$ and $Q_2(s):=g_2(w_s)$ is increasing for $\gamma_2=2$.
   
\end{lem}
\begin{proof} By Lemma \ref{l1}(iii)-(iv) it is clear that $g_1$ and $g_2$ satisfy \eqref{g1} and \eqref{g2}, respectively. In addition, Lemma \ref{l1}(i) yields that $Q_1$ is increasing.  
Finally,
\begin{eqnarray*}
    \int_\Omega|\nabla w_{s_2}|^2dx- \int_\Omega|\nabla w_{s_1}|^2dx&=&\int_\Omega \nabla ( w_{s_2}+ w_{s_1})\nabla ( w_{s_2}- w_{s_1}) dx\\&=&\int_\Omega( s_2f(w_{s_2})+s_1f(w_{s_1}))(w_{s_2}-w_{s_1})dx>0,
\end{eqnarray*}
where the last inequality follows from Lemma \ref{l1}(i) and $f>0$. Thus $Q_2$ is increasing in this case as well.
\end{proof}

\begin{proof}[Proof of Theorem \ref{thm1}]: It follows from Lemma \ref{hlemma} and Theorem \ref{thm-1}.
\end{proof}

\begin{proof}[Proof of Theorem \ref{thm3}]:
Let $v$ be a positive solution of the problem
\begin{equation*}
-\Delta v = v^{p-1} \quad \text{in }\Omega,
\qquad
v=0 \quad \text{on }\partial\Omega.
\end{equation*}
As explained in Subsection \ref{pow}, $u=sv$ (with $s>0$) solves \eqref{Pg} if, and only if,
\begin{equation*}
\lambda=\left(\frac{C_v}{\alpha}\right)^{\frac{p-2}{\gamma}} a(\alpha), \quad \alpha=C_v s^{\gamma}.
\end{equation*}

Since the map $\alpha \to \alpha^{\frac{2-p}{\gamma}} a(\alpha)$ vanishes for $\alpha=t_i$ and $\alpha=t_{i+1}$, we easily deduce $(i)$, $(ii)$, and $(iii)$.
In addition, if $a$ is strictly concave then, since $1<p<2$ the map $\alpha \mapsto a(\alpha)-\lambda \left(\frac{\alpha}{C_v}\right)^{\frac{p-2}{\gamma}}$ is strictly concave as well, for any $\lambda>0$. We deduce then the exactness of solutions in $(i)$ and $(ii)$. The same conclusion holds if $\alpha \to \alpha^{\frac{2-p}{\gamma}} a(\alpha)$ is strictly concave.\\
\end{proof}

\begin{proof}[Proof of Theorem \ref{thm3'}]:
We argue as in the previous proof. Note that if $a$ is strictly concave in $(t_i,t_{i+1})$ and $\gamma+2<p<2^*$ then $\alpha \mapsto a(\alpha)-\lambda \left(\frac{\alpha}{C_v}\right)^{\frac{p-2}{\gamma}}$ is strictly concave as well.
\end{proof}

\begin{proof}[Proof of Example \ref{exa1}]: Let $t_i=i\pi$. Then it is clear that $a$ satisfies condition \eqref{a2} and is strictly concave in $(t_i,t_{i+1})$, for any $i\in\mathbb{N}\cup\{0\}$ ($i\in\mathbb{N}$ if $p>2$). We know from Theorems \ref{thm3} and \ref{thm3'} that for all $0<\lambda<\displaystyle \min_{i\in\mathbb{N}\cup\{0\}} \max_{\alpha\in[t_i,t_{i+1}]}\left\{
a(\alpha)\left(\frac{\alpha}{C_v}\right)^{\frac{2-p}{\gamma}}
\right\}$, problem \eqref{Pg} has, for any $i\in\mathbb{N}\cup\{0\}$ ($i\in\mathbb{N}$ if $p>2$), at least  two solutions $u_{1,i},u_{2,i}$ such that
\[
i\pi<g(u_{1,i}) <g(u_{2,i}) <(i+1)\pi.
\]

To conclude, note that
\begin{equation*}
    \displaystyle \min_{i\in\mathbb{N}\cup\{0\}} \max_{\alpha\in[t_i,t_{i+1}]}\left\{
a(\alpha)\left(\frac{\alpha}{C_v}\right)^{\frac{2-p}{\gamma}}
\right\}\ge \max_{\alpha\in[0,\pi]}\left\{
a(\alpha)\left(\frac{\alpha}{C_v}\right)^{\frac{2-p}{\gamma}}\
\right\}\ge \left(\frac{\pi}{2C_v}\right)^{\frac{2-p}{\gamma}}.
\end{equation*}
\end{proof}

\begin{proof}[Proof of Example \ref{exa2}]: Let $t_i=i\pi$. Since $a(\alpha)\alpha^{\frac{2-p}{\gamma}}=|\sin \alpha|$, it is clear that $\lambda_{0,i}=C_v^\frac{p-2}{\gamma}$ for every $i$. The conclusion follows from 
Theorems \ref{thm3} and \ref{thm3'}.	
\end{proof}

\section{The Regular Case}

In this section, we assume that $f$ is locally Lipschitz continuous. In this case, we have $f(w_s)\in C^{0, \alpha}(\overline{\Omega})$ and, by elliptic regularity theory,  $w_s\in C^{2, \alpha}(\overline{\Omega})$. Thus $g$ needs to be defined only in $C_{0}^2(\overline{\Omega})$. In particular, it can be given in terms of second order derivatives of $u$.

\medskip

As previously, throughout this section, we assume that $a$ and $f$ satisfy \eqref{a1} and \eqref{f1} respectively, and that $g$ is a continuous functional on $C_{0}^2(\overline{\Omega})$ satisfying $g(0)=0$ and either \eqref{g1}, \eqref{g2}, or the following condition:

\begin{enumerate}
\item[(\mylabel{g3}{$g_3$})]  $g(u_n)\to\infty$ as $|\Delta u_n(x)|\to\infty$ in a positive measure subset of $\Omega$. 
\end{enumerate}

\medskip

Under these assumptions, Lemmas \ref{lem:unified-minimizer} and \ref{l1} remain valid. Moreover, the following result also holds:

\begin{lem}\strut
\begin{enumerate}
\item[$(i)$] If $s_n\to s\in (\lambda_1/\beta, \lambda_1/\theta)$ then $w_{s_n}\to w_{s}$ in $C^2(\overline{\Omega})$.
\item[$(ii)$] $w_s\to 0$ in $C^{2}(\overline{\Omega})$ as $s\to(\lambda_1/\beta)^+$. Moreover, if $\theta>0$ then $\displaystyle \lim_{s\to(\lambda_1/\theta)^-}|\Delta w_s(x)|=+\infty$, for all $x\in\Omega$.
\end{enumerate}
\end{lem}

\begin{proof}
$(i)$ It is a straightforward consequence of the proof of Lemma \ref{l1}(ii), that $w_n:=w_{s_n}\to w_s$ in $C^1(\overline{\Omega})$ and $w_n\in C^{1, \alpha}(\overline{\Omega})$, for some $\alpha\in (0, 1)$. Now, the local Lipschitz-continuity of $f$ implies $f(w_n)\in C^{0, \alpha}(\overline{\Omega})$ and, by elliptic regularity theory, we have $w_n\in C^{2, \alpha}(\overline{\Omega})$. The result follows by compact embedding from $C^{2, \alpha}(\overline{\Omega})$ into $C^{2}(\overline{\Omega})$ and uniqueness of the limit. 
$(ii)$ To prove the first part, it is enough to argue as in the proof of item $(i)$, to conclude that $w_n\to w$ in $C^2(\overline{\Omega})$, for some $w\in C^2(\overline{\Omega})$. Now, the result follows from Lemma \ref{l1}(iii). The second part is consequence of the equality
$$
|\Delta w_n(x)|=s_n\frac{f(w_n(x))}{w_n(x)}w_n(x) \ \mbox{in} \ \Omega,
$$
combined with \eqref{f1} and Lemma \ref{l1}(iii).

\end{proof}

\begin{exa}
We present below a variety of examples illustrating different growth regimes for second order nonlocal terms, where $\|u\|$ denotes any of the norms $\|u\|_{L^p(\Omega)}$, $\|u\|_{L^\infty(\Omega)}$, $\|\nabla u\|_{L^p(\Omega)}$, $\|\nabla u\|_{L^\infty(\Omega)}$, $\|\Delta u\|_{L^p(\Omega)}$ or $\|\Delta u\|_{L^\infty(\Omega)}$ and $p\geq 1$:

\begin{enumerate}
    \item[$(i)$] $g(u) = \int_\Omega |\Delta u(x)|^q \bigl(1 + \cos^2(|\Delta u(x)|)\bigr)\,dx, \quad q>0$;

    \item[$(ii)$] $g(u)=\frac{\|u\|^q}{\left(1 + \|u\|^q\right)^{1/q}}, \quad q>1$;

    \item[$(iii)$] $g(u) = \int_\Omega \bigl(|\Delta u(x)|^q + |\Delta u(x)|^r\bigr)\,dx, \quad 0<q<r$;

    \item[$(iv)$] $g(u) = \int_\Omega |\Delta u(x)|^q \log\bigl(1 + |\Delta u(x)|\bigr)\,dx, \quad q>0$;
    
    \item[$(v)$] $g(u) = \int_\Omega \frac{|\Delta u(x)|^2}{\sqrt{1 + |\Delta u(x)|^2}}\,dx$.
\end{enumerate}
\end{exa}

\section*{Data Availability Statement}

The authors declare that no data sets were generated or analyzed during the current study.

 \end{document}